\documentclass[a4paper,12pt,english,reqno]{amsart}

\usepackage[top=2.5cm,right=2.5cm,left=2.5cm,bottom=2.5cm]{geometry}
\usepackage[utf8]{inputenc}
\usepackage[UKenglish]{babel}
\usepackage{amsmath,amsfonts,amssymb,amsthm}
\usepackage{cite}
\usepackage{esint}
\usepackage{color}
\usepackage{mathrsfs}
\usepackage{enumerate}
\usepackage{environments}
\usepackage{graphicx}
\usepackage{subfig}
\usepackage{floatrow}
\usepackage{mathtools}

\numberwithin{equation}{section}

\definecolor{commentcolour}{rgb}{0.7, 0, 0}

\newcommand*{\smfrac}[2]{{\textstyle \frac{#1}{#2}}}

\newcommand*{\norm}[1]{\left\|#1\right\|}

\def\Aint#1{\mathchoice
  {\AXint\displaystyle\textstyle{#1}}%
  {\AXint\textstyle\scriptstyle{#1}}%
  {\AXint\scriptstyle\scriptscriptstyle{#1}}%
  {\AXint\scriptscriptstyle\scriptscriptstyle{#1}}%
  \!\int}
\def\AXint#1#2#3{{\setbox0=\hbox{$#1{#2#3}{\int}$}
\vcenter{\hbox{$#2#3$}}\kern-.5\wd0}}
\def\avint{\Aint-}

\def\<{\langle}
\def\>{\rangle}

\def\R{\mathbb{R}}

\def\N{\mathbb{N}}
\def\Z{\mathbb{Z}}

\def\CC{{\rm C}}
\def\HH{{\rm H}}
\def\WW{{\rm W}}
\def\LL{{\rm L}}
\def\PP{{\rm P}}

\def\dd{{\rm d}}

\def\dx{\,\dd x}
\def\dt{\,\dd t}
\def\ds{\,\dd s}

\def\conv{\rightarrow}
\def\wconv{\rightharpoonup}

\DeclareMathOperator*{\argmin}{\mathrm{argmin}}

\def\eps{\varepsilon}

\def\subset{\subseteq}

\newcommand{\closure}[2][3]{{}\mkern#1mu\overline{\mkern-#1mu#2}} 

\def\G{\Gamma}
\DeclareMathOperator*{\Glim}{\G--\lim}

\def\Om{\Omega}

\def\A{\mathcal{A}(L)}
\def\AN{\mathcal{A}^{\eps}(L)}

\newcommand{\Da}[1]{D_{#1}}

\def\Epure{E^\eps_{\rm p}}
\def\Edef{E_{\rm d}}
\def\Edefone{\tilde{E}_{\rm d}}
\def\Eext{E^\eps_{\rm f}}

\def\Ee{E^\eps}

\def\Einf{\tilde{E}_\infty}

\def\Fe{\mathcal{F}^\eps}
\def\Fc{\mathcal{F}_0}
\def\Fonee{\mathcal{F}^\eps_1}
\def\Fone{\mathcal{F}_1}

\def\phin{\phi_1}
\def\phinn{\phi_2}
\def\psin{\psi_1}
\def\psinn{\psi_2}
\def\Phin{\Phi_1}
\def\Phinn{\Phi_2}
\def\Psin{\Psi_1}
\def\Psinn{\Psi_2}

\def\Pone{\PP^1_\eps(\Om)}

\def\Pe{\mathcal{P}_\eps}
\def\Te{\mathcal{T}^\eps}

\newsavebox{\tempbox}

\title[$\G$-expansion for a 1D model with point defect]{Gamma-expansion for a 1D Confined Lennard-Jones model with point defect}

\author{Thomas Hudson}
\address{Mathematical Institute, Oxford, OX1 3LB, UK}
\email{hudson@maths.ox.ac.uk}

\begin{document}
  \begin{abstract}
    We compute a rigorous asymptotic expansion of the energy of a point defect
    in a 1D chain of atoms with second neighbour interactions. We propose the
    Confined Lennard-Jones model for interatomic interactions, where it is 
    assumed that nearest neighbour potentials are globally convex and second
    neighbour potentials are globally concave. We derive the $\G$-limit for the
    energy functional as the number of atoms per period tends to infinity and
    derive an explicit form for the first order term in a $\G$-expansion in
    terms of an infinite cell problem.
    We prove exponential decay properties for minimisers of the energy in the
    infinite cell problem, suggesting that the perturbation to the deformation
    introduced by the defect is confined to a thin boundary layer.
  \end{abstract}
  
  \maketitle
  
  \section{Introduction}
  \label{sec:intro}
  The analysis of discrete lattice systems and their relationship to continuum 
  mechanics is currently a growing area of study within applied analysis. Many
  rigorous results have been obtained in the past ten years connecting discrete
  models with continuum limits, which are extensively surveyed in 
  \cite{BLBL07}. The most well-developed approaches have been either to apply
  $\G$-convergence\footnote{%
  For an introduction to $\G$-convergence, see 
  \cite{BraidesGconv,DalMasoGconv}.}
  to discrete energy functionals parametrised by the number of 
  atoms per unit volume in the model (see
  \cite{Schmidt:2006,AlicCic:2004,BraiGel:2002}), or to apply forms of
  the inverse function theorem to show that for a discrete energy with the same
  parameter fixed, the Cauchy--Born rule holds; i.e. for a given atomistic
  deformation and a certain range of atomic densities there exist continuum
  deformations which are close in some norm, and have a similar energy (see 
  \cite{OrtnerTheil:12,E:2007a}). 
  
  Here, we take the former approach. We build upon recent works on surface
  energies in discrete systems \cite{SSZ:2011,BraiCic:2007}, which employ
  `$\G$-development' as first defined in \cite{AnzBaldo:93}, and extensively
  discussed in \cite{BraiTru:2008}. We define energy functionals with and
  without defects, and present the Confined Lennard-Jones model for interatomic
  interactions, which we motivate with a formal analysis.
  We then investigate the scaling of the perturbation to the
  energy which is introduced by the defect. We also provide a concrete cell
  problem that may be used for explicit computation of the first-order energy,
  and show that a minimiser of this cell problem decays exponentially away from
  the defect. This allows us to conclude that minimisers of the energies with
  and without a defect are essentially the same except on a thin boundary
  layer around the defect.
  
  \subsection{Motivation}
  \label{subsec:motiv}
  The tools developed to study the relationship between atomistic and continuum
  models rely upon the high level of symmetry which is maintained after 
  deforming a crystal. However, the pure lattice behaviour is not the only 
  factor in determining the bulk properties of such materials. The last century
  saw a revolution in the materials science community, as it was realised that
  lattice defects can change the strength of an otherwise perfect crystal by
  orders of magnitude. Understanding defects, how they scale and in what
  rigorous ways one might modify the continuum approximation of crystalline
  solids to take them into account is therefore key to developing our
  understanding of how best to model and predict their behaviour.
  
  As a first step towards this goal, we consider arguably the simplest 
  crystalline defect, a dilute point defect. A point defect is an interruption
  of the pure lattice structure caused by changing an atom at some lattice site
  (called an impurity) or by inserting an atom into the structure at a point
  which is not a lattice site (called an interstitial). In 1D, impurities and
  interstitials are essentially identical when atoms are treated as point 
  particles with hard-core interactions, since atoms cannot move past one
  another, and so it is easy to modify the reference configuration to take such
  a defect into account. In higher dimensions, the two defects are qualitatively
  different (see Figure \ref{fig:defects}), with an interstitial requiring an
  additional point in the reference configuration which breaks the symmetry.
  
  \begin{figure}
  \centering
  \subfloat[2D Defects]{
    \begin{minipage}[c][0.5\width]{0.45\textwidth}
      \centering
      \includegraphics[width=0.45\textwidth]{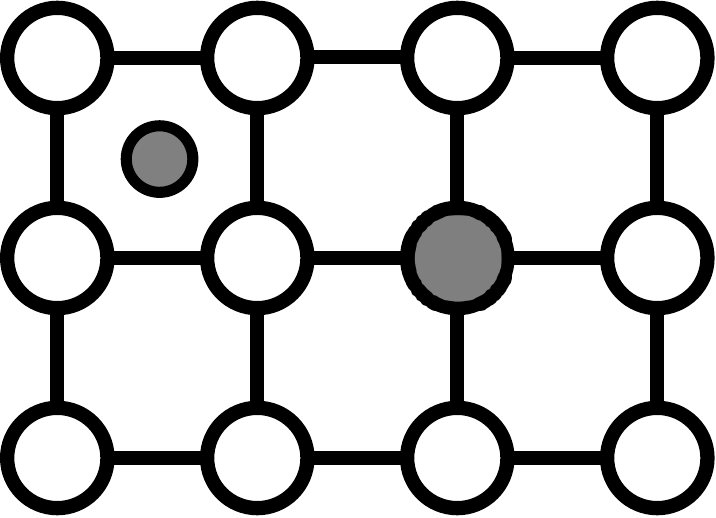}
    \end{minipage}}
  \subfloat[1D Defect]{
    \begin{minipage}[c][0.5\width]{0.45\textwidth}
      \centering
      \includegraphics[width=0.45\textwidth]{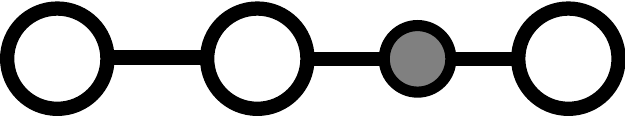}%
    \end{minipage}}
    \caption{Impurities and Interstitials.}%
    \label{fig:defects}%
  \end{figure}
  
  The model we analyse is one-dimensional, and to avoid surface effects, we use
  a periodic reference domain, so that in effect the atoms lie on a
  one-dimensional torus. The defect considered is dilute since only one atom in
  the chain is of a different type.
  
  By computing the $\G$-limit of the sequence of energy functionals for the 
  model described here, we arrive at an energy which encodes some of the
  properties of the minimisers of the functionals along the sequence, but no
  quantitative information about the error made. Computing higher-order limits 
  gives further, more quantitative control on the energy minima, and in this 
  case will also allow us to say something more qualitative about the
  minimisers.
   
  \subsection{Outline}
  \label{subsec:outline}
  As discussed above, we apply $\G$-convergence to a sequence of atomistic
  energy functionals that depend on the parameter $\eps$, which is the inverse
  of the number of atoms per unit volume.
  
  In the remainder of Section \ref{sec:intro}, we propose a model for
  interatomic interactions in a 1D chain including a point defect. We then make
  a formal analysis of the model to motivate this study and the Confined 
  Lennard-Jones model which we propose, before reformulating the problem in a
  format amenable to analysis in the framework of the one-dimensional Calculus
  of Variations.
  
  In Section \ref{sec:0}, we derive the $\G$-limit for the series of 
  functionals defined in Section \ref{sec:intro} as $\eps$ tends to zero,
  and note that the introduction of the defect does not perturb the $\G$-limit
  at this order.
  
  In Section \ref{sec:props.Ec.min}, we collect and prove some results about
  the minimum problem for the $0^\text{th}$-order $\G$-limit of the atomistic
  energies, including existence, uniqueness and regularity for minimisers.
  
  In Section \ref{sec:1}, we proceed to derive a first-order $\G$-limit,
  expressing it in terms of a minimisation problem in an infinite cell, and
  prove some properties of this minimum problem along the way.
  
  \subsection{Physical model}
  \label{subsec:model}
  For now, fix $N\in\N$. We consider $2N$ atoms indexed by 
  \begin{align*}
    i\in\{-N,\ldots,N-1\}
  \end{align*}
  which have a spatial density of $2N/L$, where $L$ is the total length of the
  deformed configuration. To make the domain periodic, we identify an atom
  indexed by $i=N$ with the atom $i=-N$ so that we avoid boundary effects, and
  defining $\eps=1/2N$, we choose to take reference positions for these atoms
  to be
  \begin{align*}
    x_i:=i\eps\in\Om_\eps:=\big\{-\smfrac12,-\smfrac12+\eps,\ldots,
    \smfrac12\big\}.
  \end{align*}
  We also define
  \begin{align*}
    \Om:=[-\smfrac12,\smfrac12],
  \end{align*}
  so that $\Om_\eps=\Om\cap\eps\Z$.
  It should be noted at the outset that the choice to use $2N$ atoms will not be
  restrictive to our analysis but will make some of the concepts easier to 
  elucidate, and that we will frequently write $\eps\conv0$ to mean
  $N\conv\infty$.
  
  Fixing the coordinate system so that atom $-N$ lies at $0$, any configuration
  can be described by a map
  \begin{align*}
    y:\Om_\eps\conv[0,L]\qquad\text{such that}\qquad y(-\smfrac12)=0,
    y(\smfrac12)=L.
  \end{align*}
  We will use the shorthand
  \begin{align*}
    y_i:=y(x_i),
  \end{align*}
  and extend $y$ to a map on the whole of $\eps\Z$ by defining
  \begin{align*}
    y_{2kN+i+1}-y_{2kN+i}:= y_{i+1}-y_i
  \end{align*}
  for any $k\in\Z$.
  
  The atoms in our model are assumed to interact through pair potentials which
  decay rapidly so that it suffices to consider an interaction between atoms and
  their 2 immediate neighbours on either side. As explained in Section
  \ref{subsec:motiv}, all atoms except one are of the same type, regarded as the
  `pure' species. As in \cite{BDMG99}, the potential energy of a bond between
  atoms is assumed to be expressed as a function of the relative displacement
  \begin{align*}
    \Da jy_i:=\frac{y_{i+j}-y_{i}}{x_{i+j}-x_i}=\frac{y_{i+j}-y_i}{j\eps}.
  \end{align*}
  In the case where all atoms are of the pure species, a bond with relative
  length $s$ has energy $\phi_1(t)$ for nearest neighbours, and $\phi_2(t)$ for
  second neighbours. The internal energy of the configuration arising from the
  interatomic forces is
  \begin{align*}
    \Epure(y):=\sum_{i=-N}^{N-1}\phin(\Da1y_i)+\phinn(\Da2y_i).
  \end{align*}
  
  Since in each configuration we assume there is a single defect, we assume
  without loss of generality that the defect is at index $i=0$. The energy of
  bonds of relative length $s$ between this atom and its neighbours are
  $\psin(t)$ for the nearest neighbours and $\psinn(t)$ for second neighbours
  (see Figure \ref{fig:chain}).
  \begin{figure}[t]
    \caption{Pair Potentials}
    \label{fig:chain}
    \includegraphics[height=3cm]{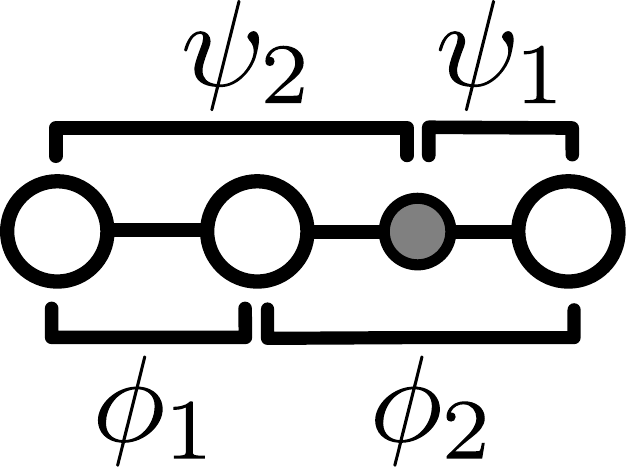}
  \end{figure}
  The introduction of the defect causes a modification of the energy which is
  given by the addition of the following energy term:
  \begin{align*}
    \Edef(y):=\psinn(\Da2y_{-2})-\phinn(\Da2y_{-2})+\psin(\Da1y_{-1})
      -\phin(\Da1y_{-1})\\
    +\psin(\Da1y_0)-\phin(\Da1y_0)+\psinn(\Da2y_0)
      -\phinn(\Da2y_0).
  \end{align*}
  Finally, we also consider dead loads $f_i$ acting on each atom.
  Taking as initial positions the points $L(x_i+\smfrac12)$, the work done by
  these forces is
  \begin{align*}
    \Eext(y):=\sum_{i=-N}^{N-1} f_i\big(y_i-L(x_i+\smfrac12)\big).
  \end{align*}
  This term can be though of as the work done by a linearisation of some
  external force field near the homogeneous linear state $y_i=L(x_i+\smfrac12)$.
  To keep notation concise, we will frequently write $u_i$ to mean
  \begin{align*}
    u_i:=y_i-L(x_i+\smfrac12),
  \end{align*}
  and we extend $f$ by periodicity to a map over $\eps\Z$ by defining
  \begin{align*}
    f_{2kN+i}=f_i
  \end{align*}
  for any $k\in\Z$.
  
  The total energy for the atomistic system considered is therefore
  \begin{align*}
    \Ee(y):=\Epure(y)+\Edef(y)+\Eext(y).
  \end{align*}
  
  \subsection{Formal analysis}
  \label{subsec:formal}
  We expect that atoms should minimise the energy $\Ee$, and we therefore seek
  to characterise the minimal energy and the states which attain this minimum.
  We will consider the situation when $N$ is large, and when the material is
  behaving elastically. In this case, interatomic displacements should vary
  slowly over the domain, and so we assume the Cauchy--Born hypothesis holds 
  (for more information, see \cite{Zanzotto96}). This states that interatomic
  displacements follow linear deformations of small volumes of the solid, and
  so we assume that
  \begin{align*}
    \Da 1y_i,\Da2y_i&\simeq Dy(x_i),
  \end{align*}
  where $y:\Om\conv[0,L]$ is some suitably smooth function describing the 
  displacement. This means that the energy
  \begin{align*}
    \phin(\Da1y_i)+\phinn(\Da2y_i)&\simeq \phin\big(Dy(x_i)\big)
      +\phinn\big(Dy(x_i)\big).
  \end{align*}
  Motivated by this we define $W$, the continuum elastic energy density to be
  \begin{align}
    W(t):=\phin(t)+\phinn(t).\label{eq:Wdef}
  \end{align}
  The total energy is now approximately
  \begin{align*}
    \Ee(y)&\simeq\sum_{i=-N}^{N-1}W(Dy(x_i))+\Edef(y)+\Eext(y).
  \end{align*}
  We expect this energy to grow linearly as the number of atoms increases, so it
  makes sense to look at the mean energy per atom, $\eps\Ee$, as $N$ gets large.
  The size of the defect is fixed and small, so $\eps\Edef(y)$ should vanish as
  $\eps\conv0$, and
  \begin{align*}
    \eps\Ee(y)&\simeq \eps\sum_{i=-N}^{N-1}W\big(Dy(x_i)\big)+f_iu_i\\
    &\simeq \Fc(y):=\int_\Om W(Dy)+fu\dx,
  \end{align*}
  where $u(x)=y(x)-L(x+\smfrac12)$.
  The minimiser of the right hand side should satisfy the Euler--Lagrange
  equation for this functional,
  \begin{align*}
    \smfrac{\rm d}{\dx}\big(W^{\prime}(Dy)\big)=f(x).
  \end{align*}
  This equation can be integrated to give
  \begin{align*}
    W^\prime(Dy)=\sigma(x)+\Sigma, \qquad\text{where}\qquad
      \sigma(x)=\int_{-1/2}^xf(t)\dt.
  \end{align*}
  The defect should then contribute a further term proportional to its size,
  $O(\eps)$, to the energy. Since the mean energy is only perturbed by a small
  amount, we should expect that any perturbation to the minimiser would also
  occur close to the defect, due to the mismatch between $\phi_i$ and $\psi_i$
  there. The defect always remains at $i=0$, so when $N$ is large, we make the 
  ansatz that close to the defect $\Da1y_i\simeq F_0+r_i$, where $F_0:=Dy(0)$
  and $r_i$ is a small perturbation. Defining
  \begin{align*}
    \sigma^\eps_{-N}&:=-\smfrac12\eps f_{-N}\qquad \text{and}\qquad
    \sigma^\eps_i:=\sigma^\eps_{i-1}+\eps f_{i-1},
  \end{align*}
  then integrating by parts, we can rewrite the external force terms as
  \begin{align*}
    \sum_{i=-N}^{N-1}f_iu_i=\sum_{i=-N}^{N-1}\sigma^\eps_i\,\Da1u_i\quad
    \text{and}\quad\int_\Om fu\dx=\int_\Om \sigma \,Du\dx.
  \end{align*}
  This then means the additional contribution is
  \begin{multline*}
    \frac{\eps\Ee(y)-\Fc(y)}{\eps}\simeq \Einf(r):=\Edefone(r)+\sum_{i=-\infty}^{\infty}
      \Big(\phin\big(F_0+r_i\big)+\phinn\big(F_0+\smfrac{r_i+r_{i+1}}{2}\big)
      - W(F_0)\\
      +\sigma_i\big(F_0+r_i-L\big)-\sigma(0)\big(F_0-L\big)\Big),
  \end{multline*}
  where we have defined
  \begin{multline*}
    \Edef(r):=\psinn(F_0+\smfrac{r_{-2}+r_{-1}}{2})
      -\phinn(F_0+\smfrac{r_{-2}+r_{-1}}{2})+\psin(F_0+r_{-1})-\phin(F_0+r_{-1})
      \\
      +\phin(F_0+r_{0})-\psin(F_0+r_0)+\psinn(F_0+\smfrac{r_{0}+r_{1}}{2})
      -\phinn(F_0+\smfrac{r_{0}+r_{1}}{2}).
  \end{multline*}
  If $f$ is smooth enough, then for $i$ close to $0$ and $\eps$ small,
  $\sigma^\eps_i\simeq\sigma(0)$, so the integrated Euler--Lagrange
  equation for $\Fc$ gives
  \begin{align*}
    \Einf(r)\simeq \Edef(r)+\sum_{i=-\infty}^{\infty}
      \phin\big(F_0+r_i\big)+\phinn\big(F_0+\smfrac{r_i+r_{i+1}}{2}\big)
      - W(F_0)-\big(W^\prime(F_0)-\Sigma\big)r_i.
  \end{align*}
  The sum of the $\Sigma\,r_i$ terms should vanish due to the boundary
  conditions. Since we expect decaying solutions, we linearise in $r_i$
  away from the defect, giving
  \begin{align*}
    \Einf(r)\simeq E_{c}(y)+\smfrac12\sum_{|i|\geq R}
      \phin^{\prime\prime}(F_0)r_i^2
      +\phinn^{\prime\prime}(F_0)\big(\smfrac{r_{i}+r_{i+1}}{2}\big)^2.
  \end{align*}
  For a minimiser of this energy, the $r_i$ should approximately satisfy
  \begin{align*}
    \smfrac12\phinn^{\prime\prime}(F_0)(r_{i-1}+r_i)
    +\phin^{\prime\prime}(F_0)r_i
    +\smfrac12\phinn^{\prime\prime}(F_0)(r_i+r_{i+1})=0.
  \end{align*}
  Making the usual ansatz $r_i=a \lambda^i$, a solution must satisfy
  \begin{align*}
    \smfrac12\phinn^{\prime\prime}(F_0)
    +W^{\prime\prime}(F_0)\lambda
    +\smfrac12\phinn^{\prime\prime}(F_0)\lambda^2=0.
  \end{align*}
  If $\phin$ and $W$ are convex and $\phinn$ is concave at $F_0$, as is the
  case in Lennard-Jones type pair potentials, then a straightforward analysis
  of the roots of this equation implies that there are two positive real roots
  which multiply to give $1$. These two roots correspond to an exponentially
  decaying solution and an exponentially growing solution.
  
  Rigorous versions of these formal results will be the subject of this paper,
  and motivate the assumptions we make about the potentials and external
  force in the following section.

  \subsection{Confined Lennard-Jones Model}
  \label{subsec:assumptions}
  Motivated by the formal analysis carried out in the previous section, this
  section details the assumptions that we make about the pair potentials and
  external force field.
  
  We will assume that all potentials $\phi_i$ and $\psi_i$ are $\CC^2$ on
  the interval $(0,\infty)$.
  Additionally, we assume the potentials and external forces satisfy the
  following conditions.
  \begin{enumerate}
    \item The nearest neighbour potentials are infinite for negative bond 
      lengths and blow up as bond lengths approach zero, i.e.
      \begin{align*}
        \phin(t),\psin(t)&=+\infty\text{ for }s\leq0,\\
        \lim_{s\searrow0}\phin(t)=+\infty\quad&\text{and}\quad
        \lim_{s\searrow0}\psin(t)=+\infty.
      \end{align*}
    \item The nearest neighbour potentials are $l$-convex, i.e. for any $s>0$,
      \begin{align*}
        \phin^{\prime\prime}(t),\psin^{\prime\prime}(t)\geq l>0.
      \end{align*}
    \item The second neighbour potentials $\phinn$ and $\psinn$ are concave.
    \item The second neighbour potentials $\phinn$ and $\psinn$ are `dominated'
      by the nearest neighbour potentials $\phin$ and $\psin$, i.e. there exist
      constants  $\alpha\in(0,1)$ and $C\in\R$ such that
      \begin{align*}
        \phinn(x)&\geq-\alpha\,\phin(x)+C, &\phinn(x)&\geq-\alpha\,\psin(x)+C,\\
        \psinn(x)&\geq-\alpha\,\phin(x)+C, &\psinn(x)&\geq-\alpha\,\psin(x)+C.
      \end{align*}
    \item The `pure' potentials are such that the resulting continuum elastic 
      potential is $l$-convex, i.e. defining $W$ as in \eqref{eq:Wdef}, for any
      $t>0$,
      \begin{align*}
        W^{\prime\prime}(t)=\phin^{\prime\prime}(t)+\phinn^{\prime\prime}(t)
        \geq l.
      \end{align*}
    \item We assume $f_i=f(x_i)$ where $f\in\CC^2(\Om)$.
  \end{enumerate}
  
  \begin{figure}[t]
    \caption{Possible choices of $\phin$ and $\phinn$ with the assumptions
    prescribed which approximate a Lennard-Jones potential,
    $\phi_{\mathrm{LJ}}$.}
    \label{fig:LJpots}
    \includegraphics[width=0.5\textwidth]{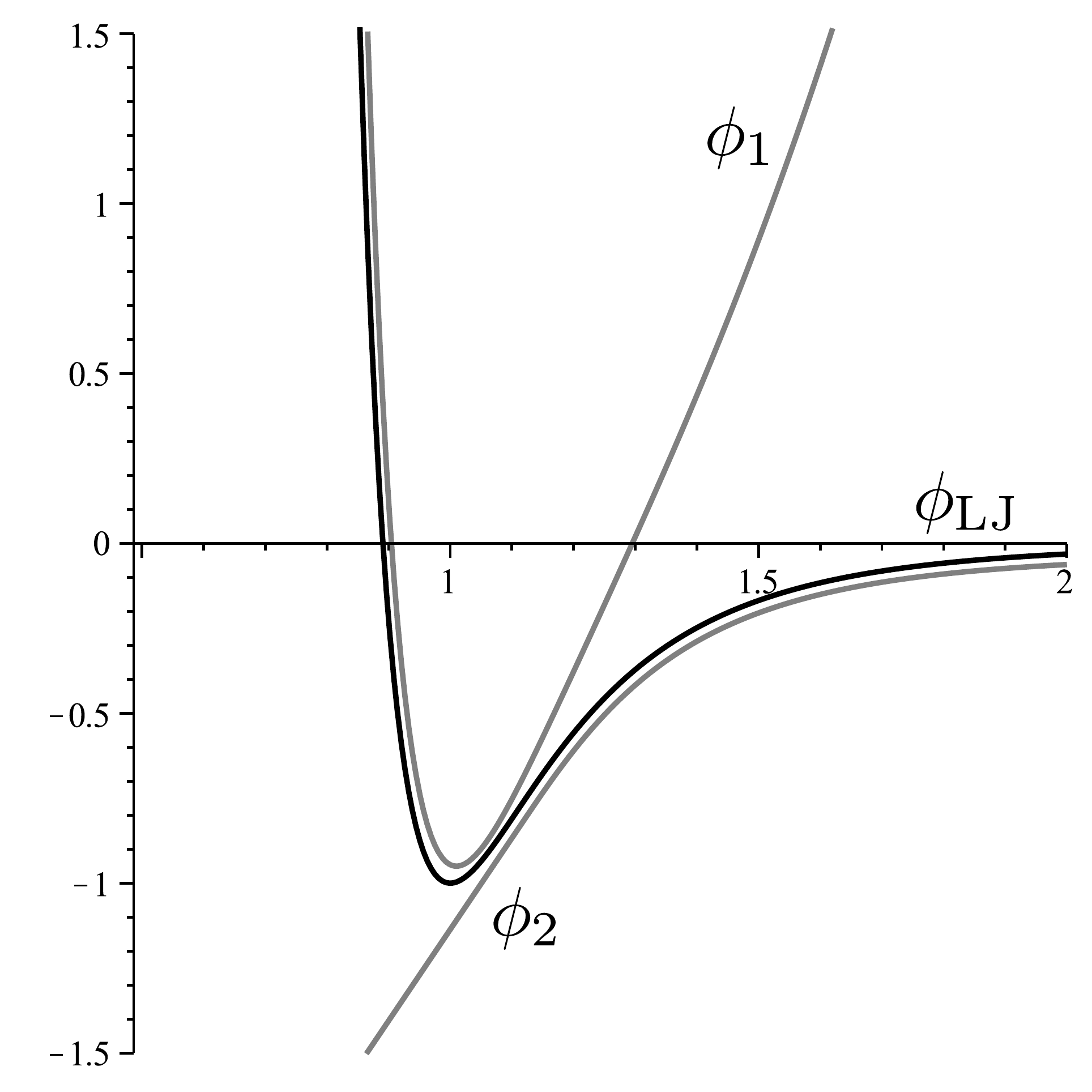}
  \end{figure}
  \medskip

  \begin{remark}
  Assumption (1) prevents atoms from exchanging positions with respect to the 
  reference configuration, and ensures that we prevent plastic deformation.
  
  Assumptions (2) and (3) are made to simulate the behaviour of a Lennard-Jones
  type potential, which is convex for short bond lengths, and then concave after
  for any bond length past some critical length; see for example Figure
  \ref{fig:LJpots}. When under strains in the elastic regime, the nearest
  neighbours lie in the convex part of the potential, and all other atoms lie
  in the concave part.
  
  Assumption (4) enforces the elastic behaviour of the material and prevents
  fracture from being favourable.
  
  Assumption (5) prevents any form of microstructure from forming, but since the
  decay of Lennard-Jones potentials used in applications is always relatively
  rapid, this assumption is reasonable, and for the sake of clarity we avoid
  significant complications to our analysis.
  \end{remark}

  These assumptions lead to the following facts which we will use frequently
  throughout this paper.
  The fact that $\phinn$ is concave implies that
  \begin{align}
    \smfrac12\phin(a)+\smfrac12\phin(b)+\phinn\big(\smfrac{a+b}{2}\big)
      \geq\smfrac12W(a)+\smfrac12W(b). \label{eq:pwlowerbnd}
  \end{align}
  By using the concavity of the second neighbour potentials again, we have
  \begin{multline*}
    \smfrac12\phin(a)+\psinn(\smfrac{a+b}{2})+\psin(b)+\phinn(\smfrac{b+c}{2})
      +\psin(c)+\psinn(\smfrac{c+d}{2})+\smfrac12\phin(d)\\
    \geq\smfrac12\big(\phin(a)+\psinn(a)\big)
      +\big(\psin(b)+\smfrac12\phinn(b)+\smfrac12\psinn(b)\big)\\
    +\big(\psin(c)+\smfrac12\psinn(c)+\smfrac12\phinn(c)\big)
      +\smfrac12\big(\phin(d)+\psinn(d)\big).
  \end{multline*}
  Assumption (4), that the behaviour of nearest neighbour potentials
  dominates, implies that
  \begin{multline}
    \smfrac12\phin(a)+\psinn(\smfrac{a+b}{2})+\psin(b)+\phinn(\smfrac{b+c}{2})
      +\psin(c)+\psinn(\smfrac{c+d}{2})+\smfrac12\phin(d)\\
    \geq(1-\alpha)\Big(\smfrac12\phin(a)+\psin(b)+\psin(c)+\smfrac12
      \phin(d)\Big)+3C,\\
    \geq \smfrac12l\,\Big(\smfrac12a^2+b^2+c^2+\smfrac12d^2\Big)+3C,
      \label{eq:defassumbd}
  \end{multline}
  where on the last line we have adjusted the definition of $l$ to keep 
  estimates concise throughout this paper.
  This estimate will allow us to prove coercivity results which ensure that
  sequences of deformations with uniformly bounded energies are compact.
  
  \subsection{Function spaces and topologies}
  \label{subsec:topology}
  In this section we define the topologies with respect to which we will carry
  out our analysis.

  Throughout this paper, we use the usual notation for Lebesgue and Sobolev
  spaces, and use $\norm{\,\cdot\,}_p$ and $\norm{\,\cdot\,}_{1,p}$ to denote
  the usual norms on $\LL^p(\Om)$ and $\WW^{1,p}(\Om)$. We write
  $y^\eps\conv y$ in $\LL^p$ or $y^\eps\conv y$ in $\WW^{1,p}$ (or $\HH^1$) to
  mean
  \begin{align*}
    \norm{y^\eps-y}_p\conv0\quad\text{or}
    \quad\norm{y^\eps-y}_{1,p}\conv 0
  \end{align*}
  respectively. We will also refer to convergence in the weak topology on
  $\HH^1(\Om)$; we say $y_\eps\wconv y$ or $y_\eps$ converges weakly to $y$ in
  $\HH^1$ if for any $f\in\HH^{-1}(\Om)$,
  \begin{align*}
    \langle f,y^\eps\rangle\conv\langle f,y\rangle,
  \end{align*}
  where $\langle\,\cdot\,,\cdot\,\rangle$ is the usual inner product on
  $\HH^1(\Om)$.
  
  Since we wish to find the $\G$-limit of the sequence of
  energy functionals $\eps \Ee$ defined above, we need to ensure they are
  defined over the same space. Following Braides, Dal Maso and Garroni in 
  \cite{BDMG99}, we associate any discrete deformation $y^\eps$ with a
  piecewise linear interpolant defined everywhere on $\Om$,
  \begin{align*}
    y^\eps(x):=y^\eps_i+\Da1y^\eps_i(x-x_i)\quad\text{for any}\quad x\in
      (x_i,x_{i+1}).
  \end{align*}
  For each choice of $\eps$, these linear interpolants lie in the spaces
  \begin{align*}
    \Pone:=\big\{y\in\WW^{1,\infty}(\Om):y\text{ is linear on }(x_i,x_{i+1}),\,
      x_i,x_{i+1}\in\Om_\eps\big\}.
  \end{align*}
  The admissible deformations $\AN$ are defined to be the set of such
  interpolants with the correct boundary conditions:
  \begin{align*}
    \AN:=\big\{y\in\Pone:y(-\smfrac12)=0,y(\smfrac12)=L\big\}.
  \end{align*}
  For any $\eps$,
  \begin{align*}
    \AN\subset\A:=\big\{y\in\HH^1(\Om):y(-\smfrac12)=0,y(\smfrac12)=L\big\};
  \end{align*}
  in fact, in the weak topology on $\HH^1(\Om)$, it is well-known that the
  sequential closure
  \begin{align*}
    \closure{\bigcup_{N=1}^\infty\AN}=\A.
  \end{align*}
  In Section \ref{subsec:0liminf}, we will show that this topology arises
  from the assumptions made in Section \ref{subsec:assumptions}.
  
  We are now in a position to suitably extend the functionals $\eps \Ee$
  so that we can take a $\G$-limit. Guided by other work for similar models
  (amongst others, see those used in \cite{BraiCic:2007,BraiGel:2002,BDMG99})
  we define $\Fe:\A\conv\R$ to be
  \begin{align*}
    \Fe(y):=\begin{cases}
      \eps \Ee(y) &y\in\AN\\
      +\infty &\text{otherwise.}
    \end{cases}
  \end{align*}

  \section{The $0^\text{th}$-order $\G$-limit}
  \label{sec:0}
  Our first result gives the first term in the $\G$-expansion of $\Fe$.
  \begin{theorem}[$0^\text{th}$-order $\G$-limit]
  \label{thm:0Gconv}
  With respect to convergence in $\LL^2$,
  \begin{align*}
    \Glim_{\eps\conv0}\Fe(y)=\Fc(y):=\int_\Om W(Dy)+f\big(y-L(x+\smfrac12)\big)
    \dx.
  \end{align*}
  \end{theorem}
  
  The $\G$-convergence of the internal energy in this result is already covered 
  by Theorem~3.1 in \cite{BraiCic:2007}, but we present a complete proof
  here in order to demonstrate the special structure of the Confined
  Lennard-Jones model described in Section \ref{subsec:assumptions}. By
  exploiting the convexity and concavity of the potentials, we do not have to
  resort to a homogenisation formula to prove the liminf inequality.
  
  As with any $\G$-convergence result, we need to prove the relevant liminf and
  limsup inequalities. We present the proofs of these inequalities in turn.
  
  \subsection{The liminf inequality}
  \label{subsec:0liminf}
  The liminf inequality is the following statement:
  \begin{proposition}
  \label{prop:0liminf}
  If $y^\eps\conv y$ in $\LL^2$ then
  \begin{align*}
    \liminf_{\eps\conv0}\Fe(y^\eps)\geq\Fc(y).
  \end{align*}
  \end{proposition}
  
  \medskip
  To prove this, we use the fact that if
  \begin{align*}
    \sup_{\eps>0}\Fe(y^\eps)<+\infty\qquad\text{and}\qquad y^\eps\conv y
    \text{ in }\LL^2,
  \end{align*}
  then $y^\eps\wconv y$. This equicoercivity result is encoded in Lemma
  \ref{lem:Lp=>weak}. We then prove that $\Fe(y^\eps)$ is approximately bounded
  below by $\Fc(y^\eps)$ for any given deformation $y^\eps\in\AN$, and finally
  we can use the fact that $\Fc$ is lower semicontinuous with respect to weak
  convergence in $\HH^1$ to obtain the inequality required.
  
  \begin{lemma}[Weak $\HH^1$ coercivity]
  \label{lem:Lp=>weak}
  If $y^\eps\conv y$ in $\LL^2$ and $\Fe(y^\eps)$ is uniformly bounded for all
  $\eps>0$, then $y^\eps\wconv y$ in $\HH^1$.
  \end{lemma}
  
  \medskip
  
  The proof of this result relies upon the growth assumptions and estimates
  made in Section \ref{subsec:assumptions}, and is inspired by the argument used
  in the proof of Theorem 4.5 in \cite{BraidesGconv}.
  
  \begin{proof}
  First, we estimate the energy below away from the defect. For ease of reading,
  let
  \begin{align*}
    P_\eps:=\{-N,\ldots,N-1\}\setminus\{-2,-1,0\},
  \end{align*}
  i.e. the set of indices which have only pure interactions with their two 
  neighbours on the right.
  We now use the inequalities from the end of Section \ref{subsec:assumptions}.
  The estimate made in \eqref{eq:pwlowerbnd}, and the $l$-convexity of $W$
  imply that for any $y\in\AN$
  \begin{align}
    \eps\sum_{i\in P_\eps}\smfrac12\phin(\Da1y_i)+\phinn(\Da2y_i)
      +\smfrac12\phin(\Da1y_{i+1})&\geq\eps\sum_{i\in P_\eps}
        \smfrac12W(\Da1y_i)+\smfrac12W(\Da1y_{i+1}),\notag\\
      &\geq\eps\sum_{i\in P_\eps}\smfrac14l\,\Big(\big|\Da1y_i\big|^2
      +\big|\Da1y_{i+1}\big|^2\Big)+C.
      \label{eq:bulkCB}
  \end{align}
  The estimate made in \eqref{eq:defassumbd} then allows us to bound the energy
  coming from bonds near the defect below.
  \begin{multline}
    \smfrac12\phin(\Da1y_{-2})+\psinn(\Da2y_{-2})+\psin(\Da1y_{-1})
      +\phinn(\Da2y_{-1})+\psin(\Da1y_0)\\
    +\psinn(\Da2y_0)+\smfrac12\phin(\Da1y_1)\geq \sum_{i=-2}^0
      \smfrac14l\,\Big(|\Da1y_i|^2+|\Da1y_{i+1}|^2\Big)+C.
      \label{eq:defbddbelow}
  \end{multline}
  Combining these estimates, we have that
  \begin{align}
    \eps\Epure(y^\eps)+\eps\Edef(y^\eps) &\geq \eps \sum_{i=-N}^{N-1}
      \Big(l\,\big|\Da1y_i^\eps\big|^2+C\Big)\notag\\
    &=l\,\norm{Dy^\eps}_2^2+C.\label{eq:bulkbddbelow}
  \end{align}
  Finally, Lemma 5.3 in \cite{BDMG99} implies that
  \begin{align*}
    \eps\Eext(y^\eps)\conv\int_\Om f\big(y-L(x+\smfrac12)\big)\dx,
  \end{align*}
  so it follows that $\eps\Eext(y^\eps)$ is uniformly bounded. Combining this 
  fact with estimate \eqref{eq:bulkbddbelow}, the uniform bound on $\Fe$ implies
  \begin{align*}
    \sup_{\eps>0}\norm{Dy^\eps}_2<+\infty.
  \end{align*}
  The argument is now concluded via the standard result that $y^\eps\wconv y$ in
  $\HH^1$ if and only if $\norm{y^\eps}_{1,2}$ is uniformly bounded and 
  $y^\eps\conv y$ in $\LL^2(\Om)$.
  \end{proof}
  
  \begin{remark}
  \label{rem:0coer}
  Lemma \ref{lem:Lp=>weak} can be interpreted as saying that if the mean energy
  is bounded along some sequence of atomistic deformations, then the 
  interatomic strains do not get too large, since they are compact in the weak
  topology on $\HH^1(\Om)$. This reinforces the notion that we are in an
  elastic regime.
  \end{remark}
  
  \medskip
  Lemma \ref{lem:Lp=>weak} permits us to use the fact that as $W$ is $l$-convex,
  the map
  \begin{align}
    y\longmapsto \int_\Om W(Dy)\dx \label{eq:contbulkenergy}
  \end{align}
  is lower semicontinuous with respect to weak convergence in $\HH^1(\Om)$ (see
  for example Corollary 2.31 in \cite{BraidesGconv}). Fix $\delta>0$. The
  estimates made in \eqref{eq:bulkCB} and \eqref{eq:defbddbelow} imply that
  \begin{align*}
    \eps\Epure(y^\eps)+\eps\Edef(y^\eps)\geq
      \int_{\Om\setminus(-\delta,\delta)}W(Dy^\eps)\dx+C\,\delta.
  \end{align*}
  Using Lemma \ref{lem:Lp=>weak} and the weak lower semicontinuity of
  \eqref{eq:contbulkenergy}, we have that
  \begin{align*}
    \liminf_{\eps\conv0}\big(\eps\Epure(y^\eps)+\eps\Edef(y^\eps)\big)&\geq
      \liminf_{\eps\conv0}\int_{\Om\setminus(-\delta,\delta)}W(Dy^\eps)\dx
      +C\delta,\\
      &\geq\int_{\Om\setminus(-\delta,\delta)}W(Dy)\dx+C\,\delta.
  \end{align*}
  Since $\delta$ was arbitrary, we let $\delta\conv0$, giving
  \begin{align}
     \liminf_{\eps\conv0}\big(\eps\Epure(y^\eps)+\eps\Edef(y^\eps)\big)
       \geq\int_\Om W(Dy)\dx.\label{eq:intenergyconv}
  \end{align}
%  Let $y^{N_k}$ be a subsequence of $y^\eps$ such that
%  \begin{align*}
%    \lim_{k\conv\infty} \int_{\Om\setminus\omega}W(Dy^{N_k})\dx =
%    \liminf_{N\conv\infty}\int_{\Om\setminus\omega}W(Dy^\eps)\dx.
%  \end{align*}
%  By taking a further subsequence if necessary, we may also assume that
%  \begin{align}
%    \bigg|\bigcup_{i=0}^\infty\omega_{N_i}\bigg|\leq 
%      \sum_{i=0}^\infty|\,\omega_{N_i}|< \infty,\label{eq:meascontdefset}
%  \end{align}
%  because $|\,\omega_N|=\delta$, and $\delta\conv0$ as $N\conv\infty$.Let $U_k$
%  be the decreasing sequence of sets
%  \begin{align*}
%    U_j:=\bigcup_{i=j}^\infty\,\omega_{N_i}. 
%  \end{align*}
%  It then follows that
%  \begin{align*}
%    \int_{\Om\setminus\omega_{N_k}}W(Dy^{N_k})\dx\geq\int_{\Om\setminus U_j}
%      W(Dy^{N_k})\dx+C|U_j|
%  \end{align*}
%  as long as $j\leq k$, where $C$ is a lower bound for $W$.
%  Fixing $j$, it follows that
%  \begin{align*}
%    \lim_{k\conv\infty}\int_{\Om\setminus\omega_{N_k}}W(Dy)\dx\geq
%      \int_{\Om\setminus U_j} W(Dy)\dx+C|U_j|,
%  \end{align*}
%  since the map in \eqref{eq:contbulkenergy} is weakly lower semicontinuous.
%  Noting that $\chi_{\Om\setminus U_j}\wsconv \chi_{\Om}$ in $\LL^\infty$, and
%  $|U_j|\conv0$ as $j\conv\infty$, both of which follow from
%  \eqref{eq:meascontdefset}, we have that
%  \begin{align}
%    \liminf_{N\conv\infty}\int_{\Om\setminus\omega_N}W(Dy^\eps)\dx \geq 
%    \int_\Om W(Dy)\dx.\label{eq:intenergyconv}
%  \end{align}
  
  Finally, convergence of the external force term is a consequence of Lemma 5.3
  in \cite{BDMG99}. This result implies that if $y^\eps\conv y$ in $\LL^2(\Om)$
  and $\norm{Dy^\eps}_2$ is uniformly bounded, then we have that
  \begin{align}
    \eps\Eext(y^\eps)\conv \int_\Om f\big(y-L(x+\smfrac12)\big)\dx,
      \label{eq:extforceconv}
  \end{align}
  and combining \eqref{eq:intenergyconv} and \eqref{eq:extforceconv} proves
  Proposition \ref{prop:0liminf}.
  
  \subsection{The limsup inequality}
  \label{subsec:0limsup}
  Now we have obtained the liminf inequality, we need to prove the limsup 
  inequality to complete the proof of Theorem \ref{thm:0Gconv}.
  
  \begin{proposition}
  \label{prop:0limsup}
  For any $y\in\A$, there exists a sequence of $y^\eps\in\AN$ such that 
  $y^\eps\conv y$ in $\LL^2$ and
  \begin{align*}
    \limsup_{\eps\conv0}\Fe(y^\eps)\leq\Fc(y).
  \end{align*}
  \end{proposition}
  
  \medskip
  The construction of the sequence $y^\eps$ requires a diagonal argument which
  is similar in flavour to that employed in the proof of Theorem 4.5 in
  \cite{BraidesGconv}. This argument proceeds in two steps. In the first step,
  the convexity of $W$ is exploited to show that a naive approximation of
  $y\in\A$ by $\Te y$ works for the `pure' part of the energy. By linearising
  deformations near the defect, and therefore controlling the behaviour of the
  energy there, we can take a diagonal sequence to arrive at the correct
  inequality.
  Note that the inequality is trivial if $\Fc(y)=+\infty$, so we only need
  consider $y\in\A$ such that $\Fc(y)<+\infty$.
  
  Fix $y\in\A$, and define the integrand
  \begin{align*}
    W_\eps(x):=\sum_{i=-N}^{N-1}\Big(\smfrac12\phinn(\Da2y_{i-1})+
      \phin(\Da1y_i)+\smfrac12\phinn(\Da2y_i)\Big)\cdot\chi_i(x),
  \end{align*}
  where $\chi_i(x)$ is the indicator function for the interval $(x_i,x_{i+1})$.
  Note that by construction,
  \begin{align*}
    \int_\Om W_\eps(x)\dx=\eps\Epure(y).
  \end{align*}
  We will apply Fatou's lemma to the functions $W_\eps$.
%  and for the interpolant $y^\eps\in\AN$ of any
%  $y\in\A$,
%  \begin{align*}
%    \Da jy^\eps_i=\frac1\eps\int_{x_i}^{x_{i+j}}Dy(x)\dx.
%  \end{align*}

  \subsubsection*{Almost everywhere convergence of $W_\eps$}
  For any $x\in\Om$, define the sequence $i_\eps:=\lfloor Nx\rfloor\in\Z$.
  Then 
  \begin{align*}
    x_\eps:=\eps i_\eps\conv x
  \end{align*}
  as $\eps\conv0$, and applying Lebesgue's Differentiation Theorem  (see for
  example Corollary 2 in Section 1.7 of \cite{EvansGariepy}) gives that for
  almost every $x\in\Om$,
  \begin{align*}
    D_jy_{i_\eps}=\avint_{x_\eps}^{x_\eps+j\eps}
      Dy(t)\dt\conv Dy(x)
  \end{align*}
  as $\eps\conv0$.
  An immediate consequence of this and the continuity of the potentials $\phi_i$
  is that
  \begin{align*}
    W_\eps(x)\conv W\big(Dy(x)\big)
  \end{align*}
  for almost every $x\in\Om$.
  
  \subsubsection*{Pointwise upper bound on $W_\eps$}
  Fix a point $x\in\Om$ and the sequence $i_\eps$ as above. Dropping the
  subscript, we estimate
  \begin{align*}
    W_\eps(x) &\leq\smfrac12\phinn(\Da2y_{i-1})
      +\smfrac12\phin(\Da2y_{i-1})+\phin(\Da1y_i)
    +\smfrac12\phin(\Da2y_i)+\smfrac12\phinn(\Da2y_i)+C,\notag\\
    &= \smfrac12W(\Da2y_{i-1})+\phin(\Da1y_i)+\smfrac12W(\Da2y_i)+C,
%    &\leq \avint_{x_{i-1}}^{x_{i+1}}W(Dy)\dx+A\;W(\Da1y_i)
%      +\avint_{x_{i}}^{x_{i+2}}W(Dy)\dx+B,\notag\\
%    &\leq C\,\avint_{x_{i-1}}^{x_{i+2}}W(Dy)\dx+B,\label{eq:dctbound}
  \end{align*}
  where $-C\in\R$ is a lower bound for $\phin$.
  Next, Assumption (4) in Section \ref{subsec:assumptions} implies that for some
  $C\in\R$,
  \begin{align*}
    \smfrac{1}{1-\alpha}W(t)+C\geq \phin(t).
  \end{align*}
  Hence, letting $A:=\max\{\smfrac12,\smfrac1{1-\alpha}\}$,
  \begin{align*}
    W_\eps(x)&\leq A\,\Big(W(\Da2y_{i-1})+W(\Da1y_i)+W(\Da2y_i)\Big)+C,\\
    &\leq A\,\bigg(\,\avint_{x_{i-1}}^{x_{i+1}}W(Dy)\dx
      +\avint_{x_i}^{x_{i+1}}W(Dy)\dx
      +\avint_{x_{i}}^{x_{i+2}}W(Dy)\dx\bigg)+C
  \end{align*}
  by using Jensen's inequality. Since $W$ is bounded below, we can extend the 
  domain of integration in each interval, possibly changing the constant $C$, 
  to reach the upper bound
  \begin{align*}
    W_\eps(x)\leq g_\eps(x):=
      3A\,\avint_{x-2\eps}^{x+2\eps}W(Dy)\dt+C.
  \end{align*}
  
  \subsubsection*{Convergence of $g_\eps$}
  As $\Fc(y)$ is bounded, $W(Dy)$ is integrable and so Lebesgue's
  Differentiation theorem implies that 
  \begin{align*}
    g_\eps(x)\conv g(x):= 3A\,W(Dy(x)) +C
  \end{align*}
  almost everywhere as $\eps\conv0$. Furthermore, $g_\eps$ is in fact the
  convolution
  \begin{align*}
    g_\eps = g\star\rho_\eps,
  \end{align*}
  where $\rho_\eps$ is an approximation to a Dirac mass given by
  \begin{align*}
    \rho_\eps(x):=\frac{1}{4\eps}\rho\bigg(\frac{x}{4\eps}\bigg),\quad
    \text{where}\quad\rho(x)=\chi_{\Om}(x).
  \end{align*}
  The functions $g_\eps$ are in $\LL^1(\Om)$ because $g\in\LL^1(\Om)$ and
  $\rho_\eps\in\LL^\infty(\Om)$, and by standard arguments
  \begin{align*}
    \int_\Om g_\eps(x)\dx\conv \int_\Om g(x)\dx
  \end{align*}
  as $\eps\conv0$. We can now apply Fatou's Lemma to $g_\eps-W_\eps$, which is
  positive, measurable, and converges almost everywhere as in  to get
  \begin{align*}
    \int_\Om g\dx -\limsup_{\eps\conv0}\int_\Om W_\eps\dx
      =\liminf_{\eps\conv0}\int_\Om g_\eps-W_\eps\dx\geq\int_\Om g-W(Dy)\dx.
  \end{align*}
  A rearrangement of this inequality allows us to conclude that
  \begin{align}
    \limsup_{N\conv\infty}\eps\Epure(y^\eps)=\limsup_{N\conv\infty}\int_\Om
      W_\eps\dx\leq \int_\Om W(Dy)\dx. \label{eq:purelimsup}
  \end{align}
  
  \subsubsection*{Controlling the energy near the defect}
  For any $y\in\A$ and $\eta>0$,
  let $y^\eta$ be a linearisation close to $0$ of $y\in\A$ given by
  \begin{align*}
    y^\eta(x):=\begin{cases}
      y(-\eta)+\frac{y(\eta)-y(-\eta)}{2\eta}(x+\eta)&\text{when }|x|<\eta,\\
      y(x) &\text{otherwise}.
    \end{cases}
  \end{align*}
  Let $F^\eta:=Dy^\eta(0)$. For $\eps$ sufficiently small, the defect energy
  for $\Te y^\eta$ is:
  \begin{align}
    \eps\Edef(\Te y^{\eta}) &= 2\eps\big(\psinn(F^\eta)-\phinn(F^\eta)
      +\psin(F^\eta)-\phin(F^\eta)\big)\leq C(\eta)\eps\label{eq:lindefconv0}
  \end{align}
  with $\eta$ fixed. To control $\eps\Eext$, we can once again employ the
  estimate that was proven in \eqref{eq:extforceconv}, since the argument used
  was for a more general sequence than that chosen here. Therefore, combining
  \eqref{eq:extforceconv}, \eqref{eq:purelimsup} and \eqref{eq:lindefconv0}, we
  deduce that
  \begin{align*}
    \limsup_{\eps\conv0}\Fe(\Te y^{\eta})\leq \Fc(y^\eta).
  \end{align*}
  Since $y^\eta\conv y$ in $\LL^2$ as $\eta\conv0$, we would like to show that
  $\Fc(y^\eta)\conv\Fc(y)$ as $\eta\conv0$ in order to use a diagonalisation
  argument. This follows from the observation that
  \begin{align*}
    2\eta\cdot C\leq\int_{-\eta}^\eta W(Dy^\eta)\dx \leq \int_{-\eta}^\eta 
      W(Dy)\dx,
  \end{align*}
  since $W$ is bounded below and convex. Both sides tend to 0 as $\eta\conv0$,
  so we can deduce that
  \begin{align*}
    \Fc(y^\eta) & = \int_\Om W(Dy^\eta) + f\big(y^\eta-L(x+\smfrac12)\big)\dx,\\
    &\leq\int_{\Om\setminus(-\eta,\eta)} W(Dy)\dx+\int_{-\eta}^{\eta}W(Dy^\eta)
      +\int_\Om f\,u\dx+\norm{f}_2\norm{y^\eta-y}_2,\\
    &\conv\Fc(y)
  \end{align*}
  as $\eta\conv0$, recalling that $u=y-L(x+\smfrac12)$.
  
  \subsubsection*{Conclusion of the argument}
  Finally, by taking a
  diagonal sequence from the collection of $\Te y^\eta$, there exists
  $\Te y^{\eta_\eps}\conv y$ in $\LL^2$, along which
  \begin{align*}
    \limsup_{\eps\conv0}\Fe(\Te y^{\eta_\eps})\leq\Fc(y),
  \end{align*}
  proving Proposition \ref{prop:0limsup}, and therefore concluding the proof
  of Theorem \ref{thm:0Gconv}.
  
  \begin{remark}
  The defect does not introduce a perturbation to the $\G$-limit at
  this order -- see Theorem 3.2 in \cite{BraiCic:2007} for the $\G$-limit of
  this problem without a defect. This is to be expected, since the `defect set'
  is null in the limit as $\eps\conv0$, and it therefore becomes reasonable to
  ask whether there is a higher order change in the energy, which is the
  subject of the subsequent analysis.
  \end{remark}

  \section{Properties of $\Fc$}
  \label{sec:props.Ec.min}
  The functional $\Fc$ is of a well-studied form, and the analysis of the 
  minimum problem is classical. The following theorem collects relevant results
  regarding the functional and its minimisers which we will invoke in the
  following sections.
  
  \begin{proposition}[Properties of $0^\text{th}$-order limit]
  The problem
  \begin{align*}
   \argmin_{y\in \A}\Fc(y)
  \end{align*}
  has a unique solution $\bar{y}$, which has the following properties:
  \begin{enumerate}
    \item $\bar{y}$ satisfies the Euler--Lagrange Equations for this
      problem,
    \item $\bar{y}\in\CC^2(\Om)$.
  \end{enumerate}
  \end{proposition}
  
  \begin{proof}
  The existence part of this proof is completely classical, and can be found in
  \cite{Dacorogna08} for example. If we suppose for the moment that minimisers
  are in $\WW^{1,\infty}(\Om)$ and satisfy the condition
  \begin{align*}
    D\bar{y}(x)\geq\delta>0
  \end{align*}
  for almost every $x\in\Om$, it is also easy to show that they satisfy
  \begin{align}
    \int_\Om \big[W^\prime\big(D\bar{y}\big)-\sigma\big]Dv\dx=0\qquad\forall\,
    v\in \WW^{1,\infty}_0(\Om)\label{eq:ELeqn},
  \end{align}
  where $\sigma(x):=\int_0^xf(t)\dt$. Since $W$ is $l$-convex,
  $W^\prime:\R^+\conv\R$ is strictly increasing and is a $\CC^1$
  diffeomorphism. If $(W^\prime)^{-1}$ is the inverse of $W^\prime$, it is
  possible to `explicitly' define a solution of the Euler-Lagrange equations
  \begin{align*}
    \bar{y}(x):=\int_{-\smfrac12}^x(W^\prime)^{-1}\big(\sigma(t)+
     \Sigma\big)\dt,
  \end{align*}
  where $\Sigma$ is the solution of the following implicit
  equation:
  \begin{align*}
      \int_{\Om}(W^\prime)^{-1}\big(\sigma(t)+\Sigma\big)\dt=L.
  \end{align*}
  We can show that this equation has a solution by regarding the left hand side
  as a function of $\Sigma$, showing it is $\CC^1$, has a strictly positive
  derivative, and tends to $0$ as $\Sigma\conv-\infty$, so attains all possible
  values $L>0$ only once. It is now simple to
  verify that $\bar{y}$ satisfies the Euler--Lagrange equation pointwise and is
  $\CC^2$, so all that remains to do is show that this is in fact the
  minimiser. Suppose that $\tilde{y}\in\A$ minimises $\Fc$ and is not equal to
  $\bar{y}$; then
  \begin{align*}
    0 &\geq \int_\Om W(D\tilde{y})-W(D\bar{y}) -f(\tilde{y}-\bar{y})\dx,\\
      &=\int_\Om W(D\tilde{y})-W(D\bar{y}) +\sigma\big(D\tilde{y}-D\bar{y}\big
        )\dx,\\
      &\geq \int_\Om \big[W^\prime(D\bar{y}) -\sigma-\Sigma\big]
        \cdot\big(D\tilde{y}-D\bar{y}\big)\dx
        +\smfrac12l\norm{D\tilde{y}-D\bar{y}}_2^2,
  \end{align*}
  where we have integrated by parts on the second line, and used the fact that
  $W$ is $l$-convex on the last line. Since $\bar{y}$ has been constructed to
  solve the Euler--Lagrange equation pointwise, the integrand vanishes, and 
  hence $\tilde{y}=\bar{y}$. This argument clearly also implies uniqueness of
  solutions.
  \end{proof}
  
  \section{The $1^\text{st}$-order $\G$-limit}
  \label{sec:1}
  The approach taken in Section \ref{subsec:0limsup} gives a strong indication
  of the scaling of the next term in an asymptotic expansion of the energy:
  \eqref{eq:lindefconv0} suggests that the extra energy from the defect is only 
  coming from a set near the defect that is of size $O(\eps)$. Section 
  \ref{sec:props.Ec.min} shows that we have a very clear understanding of the
  properties of $\bar{y}$, and thus we can reasonably hope to derive a good
  characterisation of the next order limit, as in \cite{SSZ:2011,BraiCic:2007}.
  
  For this purpose, we define some additional notation. Recalling from Section
  \ref{sec:props.Ec.min} that
  \begin{align*}
    \bar{y}=\argmin_{y\in\A}\Fc(y),
  \end{align*}
  the functional from which we obtain the first-order limit is
  \begin{align*}
    \Fonee(y):=\frac{\Fe(y)-\Fc(\bar{y})}{\eps}.
  \end{align*}
  To make the notation used in this section more concise, we let 
  $F_0:=D\bar{y}(0)$ as in Section \ref{subsec:formal}, and define potentials
  \begin{align*}
    \Phin(t)&:=\phin(F_0+t)-\phin(F_0)-\phin^\prime(F_0)\cdot t,\\
    \Psin(t)&:=\psin(F_0+t)-\phin(F_0)-\phin^\prime(F_0)\cdot t,\\
    \Phinn(t)&:=\phinn(F_0+t)-\phinn(F_0)-\phinn^\prime(F_0)\cdot t,\\
    \Psinn(t)&:=\psinn(F_0+t)-\phinn(F_0)-\phinn^\prime(F_0)\cdot t.
  \end{align*}
  We will show that the first-order $\G$-limit can be written in terms of the
  infinite cell problem
  \begin{align*}
    \inf_{r\in\ell^2(\Z)}\Einf(r),
  \end{align*}
  where $\Einf:\ell^2(\Z)\conv\R\cup\{+\infty\}$ is defined to be
  \begin{align*}
    \Einf(r):=\sum_{i=-\infty}^\infty\Phin(r_i)
      +\Phinn(\smfrac{r_{i+1}+r_i}{2}\big)+\Edefone(r),
  \end{align*}
  and we have set
  \begin{multline*}
    \Edefone(r):=\psinn\big(F_0+\smfrac{r_{-2}+r_{-1}}{2}\big)
    -\phinn\big(F_0+\smfrac{r_{-2}+r_{-1}}{2}\big)+\psin\big(F_0+r_{-1}\big)
    -\phin\big(F_0+r_{-1}\big)\\+\psin\big(F_0+r_0)-\phin\big(F_0+r_0\big)
    +\psinn\big(F_0+\smfrac{r_0+r_1}{2})-\phinn\big(F_0+\smfrac{r_0+r_1}{2}
    \big).
  \end{multline*}
  The second main result of this paper is the following theorem.
  
  \begin{theorem}[$1^\text{st}$-order $\G$-limit]
  \label{thm:1Gconv}
  With respect to convergence in $\LL^2$, we have that
  \begin{align*}
    \Glim_{\eps\conv0}\Fonee(y) = \Fone(y) :=\begin{cases}
      \displaystyle \inf_{r\in\ell^2(\Z)}\Einf(r) &y=\bar{y},\\
      +\infty &y\neq\bar{y}.
    \end{cases}
  \end{align*}
  \end{theorem}
  
  \medskip
  In contrast to the results of \cite{SSZ:2011,BraiCic:2007}, we emphasise that 
  we have an explicit representation of the $1^\text{st}$-order limit in terms
  of a minimisation problem in an infinite cell.
  Once more, the proof of this result divides into two parts, the liminf and
  limsup inequalities, which we prove in the next two sections.
  
  \subsection{The liminf inequality}
  \label{subsec:1liminf}
  The liminf inequality is the following statement.
  
  \begin{proposition}
  \label{prop:1liminf}
    If $y^\eps\conv y$ in $\LL^2$, then
    \begin{align*}
      \liminf_{\eps\conv0} \Fonee(y^\eps) \geq \Fone(y).
    \end{align*}
  \end{proposition}
  
  \medskip
  As in the proof of Proposition \ref{prop:0liminf}, we use a 
  coercivity result which says uniform boundedness of $\Fonee(y^\eps)$ implies
  a form of compactness. In this proof, there are two such results, which are
  employed at crucial steps in the main argument. The first of these results,
  Lemma \ref{lem:E1bdd=>strong}, states that if $\Fonee(y^\eps)$ is uniformly 
  bounded then the weak convergence of $y^\eps$ in $\HH^1$ proven in Lemma
  \ref{lem:Lp=>weak} improves to strong convergence in $\HH^1$. The
  second, Lemma \ref{lem:lpwkconv}, describes coercivity in a topology which
  we use to describe perturbations to the minimiser of $\Fc$ close to
  the defect. Once these results have been obtained, the main argument will
  follow by applying Fatou's Lemma to a suitable reinterpretation of
  $\Fonee(y^\eps)$.
  
  The first key step before proving the coercivity results is to rewrite
  $\Fonee(y^\eps)$ by using integration by parts on the external force terms.
  For $\Fc(\bar{y})$,
  \begin{align*}
    \int_\Om f\big(\bar{y}-L(x+\smfrac12)\big)\dx =
      -\int_\Om \sigma(D\bar{y}-L)\dx,
  \end{align*}
  using the boundary conditions, where $\sigma(x):=\int_{-1/2}^xf(t)\dt$ as in
  Section \ref{subsec:formal}. Analogously, recursively define
  \begin{align*}
    \sigma^\eps_{-N}&:=-\smfrac12\eps f_{-N},\\
    \sigma^\eps_i&:=\sigma^\eps_{i-1}+\eps f_{i-1},
  \end{align*}
  This leads to the representation
  \begin{align*}
    \sum_{i=-N}^{N-1}f_i\,u_i = \sum_{i=-N}^{N-1}\frac{\sigma^\eps_{i+1}
      -\sigma^\eps_i}{\eps}u_i=-\sum_{i=-N}^{N-1}\sigma^\eps_{i+1}\Da1u_i.
  \end{align*}
  We define the step function $\sigma^\eps:\Om\conv\R$
  \begin{align*}
    \sigma^\eps(x):=\sum_{i=-N}^{N-1}\sigma^\eps_{i+1}\,\chi_{(x_i,x_{i+1})}(x),
  \end{align*}
  so that if $y\in\AN$,
  \begin{align*}
    \eps\Eext(y)=-\int_\Om\sigma^\eps\,Du\dx.
  \end{align*}
  Using these definitions, we perform careful estimates of $\Fonee(y^\eps)$ by
  splitting the domain of integration over the intervals $(x_i,x_{i+2})$.
  For $y^\eps\in\AN$, define
  \begin{multline}
    s_i^\eps:= \phinn(\Da2y^\eps_i)+\smfrac12\phin(\Da1y^\eps_i)
    +\smfrac12\phin(\Da1y^\eps_{i+1})
      -\smfrac12\sigma_{i+1}^\eps\Da1u^\eps_i
      -\smfrac12\sigma_{i+2}^\eps\Da1u^\eps_{i+1}\\
      -\avint_{x_i}^{x_{i+2}}W(D\bar{y})-\sigma\,D\bar{u}
      +\Sigma\,\big(Dy^\eps-D\bar{y}\big)\dx.\label{eq:sidef}
  \end{multline}
  if $i\in\{-N,\ldots,N-1\}$, and set $s_i^\eps:=0$ otherwise. Then it is easy 
  to check that
  \begin{align}
    \Fonee(y^\eps)= \sum_{i=-\infty}^{\infty}s_i^\eps+\Edefone\big(\Da1y^\eps
      -F_0\big).\label{eq:altFoneedef}
  \end{align}
  We are now in a position to prove the coercivity results.
  
  \begin{lemma}[Strong $\HH^1$ coercivity]
  \label{lem:E1bdd=>strong}
  If $y^\eps\conv y$ in $\LL^2$, and $\Fonee(y^\eps)$ is uniformly bounded for
  all $\eps>0$, then $y^\eps\conv\bar{y}$ in $\HH^1$.
  \end{lemma}
  
  \begin{proof}
  Since $\Fonee(y^\eps)$ is uniformly bounded, we know that for some $C\in\R$,
  \begin{align*}
    \big|\Fe(y^\eps)-\Fc(\bar{y})\big|\leq C\eps,
  \end{align*}
  which immediately implies that $\Fe(y^\eps)\conv\Fc(\bar{y})$ as $\eps\conv0$.
  Consequently, Lemma \ref{lem:Lp=>weak} applies and so $\norm{Dy^\eps}_2$ is
  uniformly bounded. Let $i\in\{-N,\ldots,N-1\}$. We estimate $s_i^\eps$ below:
  \begin{align*}
    s_i^\eps&\geq \avint_{x_i}^{x_{i+2}}\Big(W(Dy^\eps)-W(D\bar{y})
      -\sigma^\eps Du^\eps+\sigma D\bar{u}-\Sigma\,\big(Dy^\eps-D\bar{y}\big
      )\Big)\dx,\\
    &\geq\avint_{x_i}^{x_{i+2}}\Big(W^\prime(D\bar{y})\big(Dy^\eps-D\bar{y}\big)
      +l\,|Dy^\eps-D\bar{y}|^2-\sigma^\eps Du^\eps+\sigma\,D\bar{u}-\Sigma\,
      \big(Dy^\eps-D\bar{y}\big)\Big)\dx,
  \end{align*}
  using the concavity of $\phinn$ on the first line, and the $l$-convexity of
  $W$ on the second. Next, the Euler--Lagrange equation \eqref{eq:ELeqn}
  implies that
  \begin{align}
    s_i^\eps&\geq\avint_{x_i}^{x_{i+2}}\Big(\sigma\big(Dy^\eps-D\bar{y}\big)
      +l\,|Dy^\eps-D\bar{y}|^2-\sigma^\eps Du^\eps+\sigma D\bar{u}\Big)\dx
      ,\notag\\
    &\geq\avint_{x_i}^{x_{i+2}}\Big(\big(\sigma-\sigma^\eps\big)Du^\eps
      +l\,|Dy^\eps-D\bar{y}|^2 \Big)\dx.\label{eq:siest1}
  \end{align}
  The latter term in the above integral is of the form we are looking for, so it
  now remains to show that the other term vanishes in the limit. Once this is
  done, we then show that the defect energy is also suitably bounded below.
  
  \subsubsection*{Pointwise estimate on $\sigma-\sigma^\eps$}
  Noting that $Du^\eps$ is constant on the intervals $(x_i,x_{i+1})$ and using
  the definitions of $\sigma$ and $\sigma^\eps$, we rewrite
  \begin{align*}
    \avint_{x_i}^{x_{i+1}}\big(\sigma-\sigma^\eps\big)\dx
      &=\avint_{x_i}^{x_{i+1}}\bigg( \int_{-1/2}^xf(t)\dt-\eps\sum_{j=-N}^{i-1}
      \smfrac12\big(f_j+f_{j+1}\big)-\smfrac12\eps f_i\bigg)\dx,\\
    &=\int_{-1/2}^{x_i}\Big(f(t)-\big(\Te f\big)(t)\Big)\dt
      +\avint_{x_i}^{x_{i+1}}\bigg(\int_{x_i}^xf(t)\dt
      -\smfrac12\eps f_i\bigg)\dx.
  \end{align*}
  Since we know that $f\in\CC^2$, standard results about interpolation error 
  (see for example \cite{SuliMayers}) imply that
  \begin{align}
    \bigg|\int_{-1/2}^{x_i}\Big(f(t)-\big(\Te f\big)(t)\Big)\dt\bigg|\leq 
      \smfrac{1}{12}\eps^2\norm{f^{\prime\prime}}_\infty.\label{eq:sigmaest1}
  \end{align}
  For the other term, we Taylor expand $f(t)$ at $x_i$, then evaluate integrals
  to show that
  \begin{align}
    \bigg|\,\avint_{x_i}^{x_{i+1}}\int_{x_i}^xf(t)\dt-\smfrac12\eps f_i\dx\bigg|\leq 
    \smfrac16\eps^2\norm{f^\prime}_\infty.\label{eq:sigmaest2}
  \end{align}
  Combining \eqref{eq:sigmaest1} and \eqref{eq:sigmaest2}, we have
  \begin{align}
    \bigg|\,\avint_{x_i}^{x_{i+2}}(\sigma-\sigma^\eps)Du^\eps\dx\bigg|&\leq
       \Big(\smfrac{1}{12}\eps^2\norm{f^{\prime\prime}}_\infty+
       \smfrac16\eps^2\norm{f^\prime}_\infty\Big)\frac{|\Da1u^\eps_i|
       +|\Da1u^\eps_{i+1}|}{2},\notag\\
       &\leq C\eps^2\Big(|\Da1u^\eps_i|^2+|\Da1u^\eps_{i+1}|^2\Big)^{1/2},
       \notag\\
       &\leq C\eps^{3/2}\norm{Du^\eps}_2,\label{eq:sigmaest3}
  \end{align}
  where on the second line we used Jensen's inequality, and on the third line
  we used $\eps^{1/2}$ and added further postive terms inside the brackets to
  get the estimate.
  This can be used in \eqref{eq:siest1} to give
  \begin{align}
    s_i^\eps\geq -C\eps^{3/2}\norm{Du^\eps}_2
      +l\,\avint_{x_i}^{x_{i+2}}|Dy^\eps-D\bar{y}|^2\dx.\label{eq:bulk1coercive}
  \end{align}
  
  \subsubsection*{Lower bound on defect energy}
  Using \eqref{eq:defbddbelow}, the fact that
  $W(D\bar{y})-(\sigma+\Sigma) D\bar{u}$ is finite and estimate
  \eqref{eq:sigmaest3},
  \begin{align*}
    \sum_{i=-2}^0s_i^\eps+\Edefone(Dy^\eps)&\geq \sum_{i=-2}^0\smfrac14l\,
      \Big(\big|\Da1y^\eps_i\big|^2+\big|\Da1y^\eps_{i+1}\big|^2\Big)
      +C-\avint_{x_i}^{x_{i+2}}\Big(\sigma^\eps Du^\eps
      +\Sigma\,Du^\eps\Big)\dx,\\
    &=\sum_{i=-2}^0\,\avint_{x_i}^{x_{i+2}}\Big(l\,|Dy^\eps|^2-\sigma Du^\eps
      -(\sigma^\eps-\sigma)Du^\eps-\Sigma\,Du^\eps\Big)\dx+C,\\
    &\geq\sum_{i=-2}^0\avint_{x_i}^{x_{i+2}}\Big(l\,|Dy^\eps|^2
      -\norm{\sigma+\Sigma}_\infty|Du^\eps|+C\Big)\dx
      -C\eps^{3/2}\norm{Dy^\eps}_2.
  \end{align*}
  Since $D\bar{y}$ is bounded above and below, by adjusting constants suitably
  we have that
  \begin{align}
    \sum_{i=-2}^0s_i^\eps+\Edefone(Dy^\eps)&\geq \sum_{i=-2}^0
      \avint_{x_i}^{x_{i+2}}l\,|Dy^\eps-D\bar{y}|^2\dx +C
      -C\eps^{3/2}\norm{Dy^\eps}_2.\label{eq:defect1coercive}
  \end{align}
  
  \subsubsection*{Conclusion of the argument}
  By summing over $i$ in \eqref{eq:bulk1coercive}, combining with
  \eqref{eq:defect1coercive}, and using the fact that $\norm{Dy^\eps}_2$ is
  uniformly bounded, we have shown that
  \begin{align*}
    \Fonee(y^\eps)\geq -C\eps^{1/2} + l\,\sum_{i=-N}^{N-1}
      \avint_{x_i}^{x_{i+1}}|Dy^\eps-D\bar{y}|^2\dx+C.
  \end{align*}
  Multiplying this inequality by $\eps$ and using the assumption that
  $\Fonee(y^\eps)$ is uniformly bounded, we have
  \begin{align}
    C\eps\geq l\norm{Dy^\eps-D\bar{y}}_2^2,\label{eq:1coercivity}
  \end{align}
  which proves the result.
  \end{proof}
  
  To prove the second coercivity result, we define the sequence of operators
  $\Pe:\AN\conv\ell^2(\Z)$ by
  \begin{align*}
    \big(\mathcal{P}_\eps y\big)_i:=\begin{cases}
      \Da1y_i-\Da1\bar{y}_i &i\in\{-N,\ldots,N-1\}\\
      0 &\text{otherwise}.
    \end{cases}
  \end{align*}
  Clearly $\mathcal{P}_\eps y$ is well-defined since this sequence is non-zero
  only on a finite set.
  
  \begin{lemma}[Weak $\ell^2(\Z)$ coercivity]
  \label{lem:l2.coercivity}
  If $\Fonee(y^\eps)$ is uniformly bounded, then there exists a subsequence of
  $\mathcal{P}_\eps y^\eps$ which converges weakly in $\ell^2(\Z)$.
  \end{lemma}
  
  \begin{proof}
  By dividing \eqref{eq:1coercivity} by $\eps$ and using Jensen's inequality,
  we have
  \begin{align*}
    C\geq l\sum_{i=-N}^{N-1}\avint_{x_i}^{x_{i+1}}|Dy^\eps-D\bar{y}|^2\dx
      \geq l \sum_{i=-N}^{N-1}\bigg|\,\avint_{x_i}^{x_{i+1}}Dy^\eps-D\bar{y}\dx
      \bigg|^2=l\norm{\Pe y^\eps}_{\ell^2(\Z)}^2.
  \end{align*}
  We have shown that the sequence $\Pe y^\eps$ is uniformly bounded in
  $\ell^2(\Z)$, so in particular, it must have a weakly convergent subsequence.
  \end{proof}
  
  To conclude the argument which will prove the liminf inequality, we will use
  the following characterisation of weak convergence in $\ell^2(\Z)$ which
  follows easily from the Riesz Representation Theorem.
  
  \begin{lemma}
  \label{lem:lpwkconv}
  A sequence $(r^\eps)\subset\ell^2(\Z)$ converges weakly to $r\in\ell^2(\Z)$
  as $\eps\conv0$ if and only if the following two conditions hold:
  \begin{enumerate}
    \item $\norm{r^\eps}_{\ell^2}$ is uniformly bounded,
    \item $r^\eps\conv r$ pointwise (almost everywhere in the counting 
          measure) as $\eps\conv0$.
  \end{enumerate}
  \end{lemma}
  
  \medskip
  
  As indicated at the beginning of this section, we apply Fatou's Lemma to the
  sum \eqref{eq:altFoneedef}. Suppose that $\Fonee(y^\eps)$ is uniformly bounded
  and $y^\eps\conv y$. Take a subsequence $y^{\eps_k}$ such that
  \begin{align*}
    \lim_{k\conv\infty}\mathcal{F}_1^{\eps_k}(y^{\eps_k})=\liminf_{\eps\conv0}\Fonee(y^\eps),
  \end{align*}
  and then using Lemma \ref{lem:l2.coercivity}, a further subsequence (which we 
  do not relabel) such that $\mathcal{P}_{\eps_k}y^{\eps_k}$ weakly converges to
  $r$ in $\ell^2(\Z)$. Since $\bar{y}\in\CC^2(\Om)$, we have that
  \begin{align*}
    \Da1\bar{y}_i=\avint_{x_i}^{x_{i+1}}D\bar{y}\dx\conv F_0:=D\bar{y}(0)
  \end{align*}
  as $\eps\conv0$. Fixing an index $i\in\Z$, Lemma \ref{lem:lpwkconv} 
  implies that
  \begin{align*}
    \big(\mathcal{P}_{\eps_k}y^{\eps_k}\big)_i&=\Da1y^{\eps_k}_i
      -\Da1\bar{y}_i,\\
    &\conv F_0+r_i-F_0
  \end{align*}
  as $k\conv\infty$, so that we may view $r$ as a perturbation to the 
  deformation gradient in an `infinitesimal' neighbourhood of the defect.
  The `pointwise' estimate \eqref{eq:bulk1coercive} implies that for
  $i\in\{-N,\ldots,N-1\}$
  \begin{align*}
    s_i^\eps+C\eps^{3/2}\geq0,
  \end{align*}
  so that
  \begin{align}
    \liminf_{k\conv\infty} \sum_{i=-N_k}^{N_k-1}s^{\eps_k}_i+C\eps_k^{3/2}&\geq
      \sum_{i=-\infty}^\infty\liminf_{k\conv\infty}\big(s_i^{\eps_k}
      +C\eps_k^{3/2}\big).\label{eq:liminf1fatou}
  \end{align}
  Since the potentials $\phi_i$ are continuous and $\sigma^\eps_i\conv\sigma(0)$
  as $\eps\conv0$ with $i$ fixed, we have that
  \begin{multline*}
    \liminf_{k\conv\infty}\big(s_i^{\eps_k}+\eps_k^{3/2}\big)
      =\phinn(F_0+\smfrac{r_i+r_{i+1}}{2})
      +\smfrac12\phin(F_0+r_i)+\smfrac12\phin(F_0+r_{i+1})
      -W(F_0)\\-\sigma(0)(F_0+\smfrac{r_i+r_{i+1}}{2})+\sigma(0)F_0
      -\Sigma\,\smfrac{r_i+r_{i+1}}{2}.
  \end{multline*}
  Recalling that $\bar{y}$ satisfies the Euler--Lagrange equations pointwise,
  we have that
  \begin{align*}
    \sigma(0)=W^\prime(F_0)-\Sigma,
  \end{align*}
  so that
  \begin{align}
    \liminf_{k\conv\infty}\big(s_i^{\eps_k}+\eps_k^{3/2}\big)
      =\Phinn(r_i+r_{i+1})+\smfrac12\Phin(r_i)+\smfrac12\Phin(r_{i+1}).
      \label{eq:liminf1pw}
  \end{align}
  Note that
  \begin{align}
    \lim_{k\conv\infty}\sum_{i=-N_k}^{N_k-1}C\eps_k^{3/2}
      =\lim_{k\conv\infty}C\eps_k^{1/2}=0,\label{eq:liminf1pert}
  \end{align}
  so then combining \eqref{eq:liminf1fatou}, \eqref{eq:liminf1pw} and
  \eqref{eq:liminf1pert}, we have that
  \begin{align*}
    \liminf_{k\conv\infty}\sum_{i=-\infty}^{\infty} s_i^{\eps_k}
      \geq\sum_{i=-\infty}^\infty\Phinn\big(\smfrac{r_i+r_{i+1}}{2}\big)
      +\smfrac12\Phin(r_i)+\smfrac12\Phin(r_{i+1}).
  \end{align*}
  Finally, by possibly taking further subsequences, we can assume that
  $\big(\mathcal{P}_{\eps_k}y^{\eps_k}\big)_i$ converges uniformly for
  $i\in\{-2,-1,0,1\}$, and then we have
  \begin{align*}
    \liminf_{\eps\conv0}\Fonee(y^\eps)&=
      \liminf_{k\conv\infty}\bigg(\sum_{i=-\infty}^{\infty}s_i^{\eps_k}
      +\Edefone(Dy^{\eps_k})\bigg),\\
    &\geq\Einf(r),\\
    &\geq\inf_{r\in\ell^2(\Z)}\Einf(r),
  \end{align*}
  proving Proposition \ref{prop:1liminf}.
  
  \begin{remark}
  By definition, we have 
  \begin{align*}
    \sum_{i=-N}^{N-1} \big(\Pe y^\eps\big)_i=0
  \end{align*}
  for any $y^\eps\in\AN$. If $\Pe y^\eps\wconv r$ in $\ell^1(\Z)$, we could
  conclude that
  \begin{align*}
    \sum_{i=-\infty}^{\infty}r_i=0;
  \end{align*}
  however, since we have convergence only in $\ell^2(\Z)$, this is not true in
  general. It is therefore clear that the set of compactly supported mean zero
  sequences is dense in the weak topology on $\ell^2(\Z)$.
  \end{remark}

  \subsection{The limsup inequality}
  \label{subsec:1limsup}
  The limsup inequality is the following statement.
  
  \begin{proposition}
  \label{prop:1limsup}
    For every $y\in\A$, there exists a sequence $y^\eps\conv y$ in $\LL^2$ such
    that
    \begin{align*}
      \limsup_{\eps\conv0} \Fonee(y^\eps) \leq \Fone(y).
    \end{align*}
  \end{proposition}
  
  \medskip
  This statement is trivial in the case where $y\neq\bar{y}$, so we only need to
  construct the sequence for $y=\bar{y}$. In order to construct the limsup
  sequence, we will show that there exists a minimiser of $\Einf$, and then 
  combine a suitable truncation of this minimiser with $\Te\bar{y}$ to get the
  result.
  
  \begin{proposition}[Properties of $1^\text{st}$-order limit]
  \label{prop:Einf.existence}
  Let $\Einf:\ell^2(\Z)\conv\R$ be defined as in Section \ref{sec:1}. Then
  there exist minimisers of $\Einf$ which satisfy an infinite system of
  nonlinear algebraic Euler--Lagrange equations, given in \eqref{eq:EL.1}.
  \end{proposition}
  
  \begin{proof}
  Since $\Einf(r)<+\infty$ for the constant sequence $r=0$, the infimum is less
  than $+\infty$. We show that existence follows from the direct method of the
  Calculus of Variations applied to $\Einf$. It is easy to check that $\Phinn$
  is concave because $\phinn$ is concave, hence
  \begin{align*}
    \Einf(r)\geq\sum_{i=-\infty}^{\infty}\Phin(r_i)+\Phinn(r_i)+\Edefone(r).
  \end{align*}
  Next,
  \begin{align*}
    \Phin(t)+\Phinn(t)=W(F_0+t)-W(F_0)-W^\prime(F_0)t\geq \smfrac12l\,t^2,
  \end{align*}
  using the $l$-convexity of $W$. For $i\in\{-2,-1,0,1\}$, we can estimate below
  as in \eqref{eq:defbddbelow}, so that for some constant $C\in\R$
  \begin{align*}
    \sum_{i=-2}^{1}\Big(\smfrac12\Phin(r_i)+\smfrac12\Phin(r_{i+1})
    +\Phinn\big(\smfrac{r_i+r_{i+1}}{2}\big)\Big)+\Edefone(r)\geq
    \sum_{i=-2}^{1}\smfrac12l\,|r_i|^2+C.
  \end{align*}
  Hence we have that
  \begin{align*}
    \Einf(r)\geq \smfrac12l\norm{r}_{\ell^2(\Z)}^2+C.
  \end{align*}
  Next we need to show that $\Einf$ is sequentially weakly lower semicontinuous.
  This follows by using Fatou's lemma as above with the pointwise lower bounds
  just proven. A standard application of the direct method now yields existence.
  
  To obtain the Euler--Lagrange equations, suppose $r$ is a minimiser of
  $\Einf$. Let $e^i\in\ell^2(\Z)$ be the sequence which has
  \begin{align*}
    e^i_j=\begin{cases}
    1 &j=i,\\
    0 &\text{otherwise}.
    \end{cases}
  \end{align*}
  Let $i\notin\{-2,-1,0,1\}$; for small enough $t>0$, $\Einf(r+te^i)<+\infty$,
  and
  \begin{align*}
    0&\leq\frac{\Einf(r+te^i)-\Einf(r)}{t},\\
    &=\int_0^1\smfrac12\Phinn^\prime\big(\smfrac{r_{i-1}+r_i+st}{2}\big)
      +\Phin^\prime(r_i+st)
      +\smfrac12\Phinn^\prime\big(\smfrac{r_i+st+r_{i+1}}{2}\big)\ds.
  \end{align*}
  Applying the Dominated convergence theorem and repeating the argument for
  $t<0$ now implies that
  \begin{subequations}
  \begin{align}
    \smfrac12\Phinn^\prime(\smfrac{r_{i-1}+r_i}{2})+\Phin^\prime(r_i)
      +\smfrac12\Phinn^\prime(\smfrac{r_i+r_{i+1}}{2})&=0.\label{eq:EL.1.bulk}
  \end{align}
  By the same argument, we also have that
  \label{eq:EL.1}
  \begin{align}
  \smfrac12\Phinn^\prime\big(\smfrac{r_{-3}+r_{-2}}{2}\big)
    &+\Phin^\prime(r_{-2})
    +\smfrac12\Psinn^\prime\big(\smfrac{r_{-2}+r_{-1}}{2}\big)=0,\\
  \smfrac12\Psinn^\prime\big(\smfrac{r_{-2}+r_{-1}}2\big)&+\Psin^\prime(r_{-1})+
    \smfrac12\Phinn^\prime\big(\smfrac{r_{-1}+r_0}2\big)=0,\\
  \smfrac12\Phinn^\prime\big(\smfrac{r_{-1}+r_0}2\big)&+\Psin^\prime(r_0)+
    \smfrac12\Psinn^\prime\big(\smfrac{r_0+r_1}2\big)=0,\\
  \smfrac12\Psinn^\prime\big(\smfrac{r_0+r_1}2\big)&+\Phin^\prime(r_1)+
    \smfrac12\Phinn^\prime\big(\smfrac{r_1+r_2}2\big)=0,
  \end{align}
  \end{subequations}
  completing the proof.
  \end{proof}
  
  In order to complete the proof of the limsup inequality, we will require a
  better understanding of minimisers of $\Einf$, and so we prove the
  following sequence of results, which amount to regularity results for
  solutions of the Euler--Lagrange equations \eqref{eq:EL.1}.
  From now on, we fix $r\in\ell^2(\Z)$ as being one particular minimiser of
  $\Einf$.
  
  \begin{lemma}
  Suppose that $r\in\ell^2(\Z)$ solves \eqref{eq:EL.1.bulk} for all $i\geq2$.
  Then
  \begin{align*}
    r_1\geq0 \quad&\Rightarrow \quad r_1=\max_{j\in\N}\{r_j\};\\
    r_1\leq0 \quad&\Rightarrow \quad r_1=\min_{j\in\N}\{r_j\}.
  \end{align*}
  \end{lemma}
  
  \begin{proof}
  We prove only the first conclusion, the proof of the second being similar.
  Suppose for a contradiction that $r_M$ is an interior maximum, i.e. that
  \begin{align*}
    r_M=\max_{j\in\N}\{r_j\}\geq0.
  \end{align*}
  Then because $\Phinn$ is concave, we have that $\Phinn^\prime$ is monotone
  decreasing, and hence
  \begin{align}
    0&=\smfrac12\Phinn^\prime(\smfrac{r_{M-1}+r_M}{2})+\Phin^\prime(r_M)
      +\smfrac12\Phinn^\prime(\smfrac{r_M+r_{M+1}}{2}),\notag\\
    &\geq\Phinn^\prime(r_M)+\Phin^\prime(r_M),\notag\\
    &=W^\prime(F_0+r_M)-W^\prime(F_0),\notag\\
    &=\int_0^{r_M}W^{\prime\prime}(F_0+s)\ds,\notag\\
    &\geq l\,r_M,\label{eq:max.princ}
  \end{align}
  which implies that either $r_i=0$ for all $i\geq1$ or a contradiction, 
  concluding the proof.
  \end{proof}
  
  \begin{corollary}
  \label{cor:monotonicity}
  If $(r_i)_{i=1}^\infty\in\ell^2(\N)$ solves the Euler--Lagrange
  equations with $r_1$ fixed, then
  \begin{align*}
    r_1\geq0 \quad&\Rightarrow \quad r_i\geq r_{i+1}\geq0
      \text{ for all }i\in\N;\\
    r_1\leq0 \quad&\Rightarrow \quad r_i\leq r_{i+1}\leq0
      \text{ for all }i\in\N.
  \end{align*}
  \end{corollary}
  
  \begin{proof}
  Estimate \eqref{eq:max.princ} states that if
  \begin{align*}
    r_i=\max\{r_{i-1},r_i,r_{i+1}\},\quad\text{then}\quad r_i\leq 0.
  \end{align*}
  Similarly, it is possible to show that if
  \begin{align*}
    r_i=\min\{r_{i-1},r_i,r_{i+1}\},\quad\text{then}\quad r_i\geq0.
  \end{align*}
  Suppose that $r$ has a strict local maximum $r_M<0$ with $M>1$. Then
  $r_{M+1}\leq r_M<0$. Since local minima can only occur when $r_i\geq0$,
  $r_{M+1}$ cannot be a local minimum, and so $r_{M+2}\leq r_{M+1}<0$. 
  Proceeding by induction, $r_i\leq r_M<0$ for all $i\geq M$, which contradicts
  the fact that $r\in\ell^2(\N)$.
  A similar argument prevents the existence of strict local minima.
  
  Next suppose that $r_M=0$ is a local maximum for $M>1$. If $r_{M+1}<0$, then
  the previous argument applies. If $r_{M+1}=0$, then it too must be a local
  maximum or minimum, depending on the sign of $r_{M+2}$. If $r_{M+2}\neq0$,
  then we can apply the previous arguments again to arrive at a contradiction,
  so by induction we have that $r_i=0$ for all $i\geq M$.
  
  We have therefore shown that there can be no internal maxima, unless they are
  degenerate in the sense that $r$ is identically $0$ after the maximum, and by
  a similar argument, we can show that there can be no internal minima except
  if they are degenerate in the same sense. We can now conclude that any 
  solution of \eqref{eq:EL.1.bulk} must be increasing if $r_1\geq0$, or
  decreasing if $r_1\leq0$, which concludes the proof.
  \end{proof}
  
  Finally, we prove that minimisers have exponentially small `tails'.
  
  \begin{proposition}[Exponential decay]
  Let $C\geq0$ be the constant such that for all
  $t\in(-r_1,r_1)$,
  \begin{align*}
  0\geq\Phinn^{\prime\prime}(t)&\geq-C.
  \end{align*}
  Then if $\lambda:=\smfrac{C}{l+C}$, we have that
  \begin{align*}
    |r_i|\leq \lambda^{i-1}\, |r_1|.
  \end{align*}
  \end{proposition}
  
  \begin{proof}
  We will only prove the result for $r_1\geq0$, since the other case is similar.
  The Euler--Lagrange equations may be rewritten using the Fundamental Theorem
  of Calculus as
  \begin{align*}
    \int_0^{r_i}W^{\prime\prime}\big(F_0+t\big)\dt+\int_{r_i}^{r_{i-1}}
    \Phinn^{\prime\prime}(t)\dt+\int_{r_i}^{r_{i+1}}
    \Phinn^{\prime\prime}(t)\dt=0.
  \end{align*}
  Corollary \ref{cor:monotonicity} now gives that $r_{i-1}\geq r_i\geq r_{i+1}$,
  so we have
  \begin{align*}
    0&= \int_0^{r_i}W^{\prime\prime}\big(F_0+t\big)\dt+\int_{r_i}^{r_{i-1}}
    \Phinn^{\prime\prime}(t)\dt-\int_{r_{i+1}}^{r_i}\Phinn^{\prime\prime}(t)\dt,
    \\
    &\geq\int_0^{r_i}l\,\dt-\int_{r_i}^{r_{i-1}}C\,\dt,
  \end{align*}
  using the assumed bound on the second derivative of $\Phinn$, and the
  $l$-convexity of $W$. It immediately follows that
  \begin{align*}
    \frac{C}{l+C}\cdot r_{i-1}\geq r_i,
  \end{align*}
  which is true for any $i\geq2$ and the decay estimate is obtained by using
  induction on this inequality.
  \end{proof}
  
  \begin{remark}
  It should immediately be noted that since we know that $r_i$ converges to zero
  exponentially as $i\conv\pm\infty$, it must be the case that $r\in\ell^1(\Z)$.
  This will be crucial in what follows.
  \end{remark}
  
  \medskip
  We can now apply this characterisation of minimisers of $\Einf$ to complete the
  proof of Proposition \ref{prop:1limsup}. Let the sequence of functions $y^{\eps,\delta}$
  be given by
  \begin{align*}
    y^{\eps,\delta}(x):=\int_{-1/2}^xDy^{\eps,\delta}(t)\dt,
  \end{align*}
  where we have set
  \begin{align*}
    Dy^{\eps,\delta}(x)&:=\begin{cases}
      \Da1\bar{y}_i+r_i-\bar{r}^\delta &i\in\{-K,\ldots,K-1\},\\
      \Da1\bar{y}_i &\text{otherwise},
    \end{cases}\\
    \bar{r}^\delta&:=\delta\sum_{i=-K}^{K-1}r_i,\\
    \delta&:=\frac{1}{2K}.
  \end{align*}
  We will prove estimates for $y^{\eps,\delta}$, and then set $\delta$ as a
  function of $\eps$ in order to obtain the recovery sequence. Since
  $r\in\ell^1(\Z)$, we have that
  \begin{align}
    |\bar{r}^\delta|&\leq\delta\sum_{i=-K}^{K-1}|r_i|
    \leq\delta\norm{r}_{\ell^1(\Z)}. \label{eq:limsup1est1}
  \end{align}
  By construction, $y^{\eps,\delta}\in\AN$ as long as $\delta\geq\eps$, i.e.
  $K\leq N$.
  
  \subsubsection*{Ensuring $y^{\eps,\delta}$ is well-defined}
  The first step we take is to check that $W(\Da1y_i^{\eps,\delta})$ is
  well-defined. Since $\Einf(r)<+\infty$, we have that for all $i\in\Z$,
  \begin{align*}
    W(F_0+r_i)<+\infty\quad\Rightarrow\quad
      F_0+r_i\geq F_0-\norm{r}_{\ell^\infty(\Z)}>0,
  \end{align*}
  as $\ell^2(\Z)\subset\ell^\infty(\Z)$.
%  Since $\Einf(r)<+\infty$, and $\ell^2(\Z)\subset\ell^\infty(\Z)$, we have that
%  \begin{align*}
%    F_0+\norm{r}_{\ell^\infty(\Z)}\geq F_0+r_i\geq 
%      F_0-\norm{r}_{\ell^\infty(\Z)}>0.
%  \end{align*}
%    and hence there exists $C>0$ such that for all $i\in\Z$,
%  \begin{align*}
%    W^\prime(F_0+r_i)\leq C,
%  \end{align*}
  Next, by using the Mean Value Theorem and \eqref{eq:limsup1est1}, 
  we estimate for $i\in\{-K,\ldots,K-1\}$ that
  \begin{align*}
    |\Da1y^{\eps,\delta}_i-F_0-r_i|&=\bigg|\,\avint_{x_i}^{x_{i+1}}D\bar{y}\dx
      -\bar{r}^\delta-D\bar{y}(0)\bigg|,\\
    &=|D\bar{y}(\xi)-\bar{r}^{\delta}-D\bar{y}(0)|,\\
    &\leq\bigg|\int_0^\xi D^2\bar{y}(x)\dx\bigg|
      +\delta\norm{r}_{\ell^1(\Z)},\\
    &\leq|x_{i+1}|\norm{D^2\bar{y}}_\infty
      +\delta\norm{r}_{\ell^1(\Z)},\\
    &\leq \frac{\eps}{2\delta}\norm{D^2\bar{y}}_\infty
      +\delta\norm{r}_{\ell^1(\Z)}.
  \end{align*}
  On the second line, $\xi$ is some point in $(x_i,x_{i+1})$, and on the final
  line we have used $|x_{i+1}|\leq K\eps =\smfrac{\eps}{2\delta}$.
  For fixed $i$, it is now clear that
  \begin{align*}
    \big|W(\Da1y_i^\eps)-W(F_0+r_i)\big|\leq C \sup\big\{|W^\prime(t)|:t\in(F_0
      +r_i,\Da1y_i^\eps)\big\}\big(\smfrac{\eps}{\delta}+\delta\big),
  \end{align*}
  therefore $W(\Da1y^\eps_i)$ is finite when $\eps\ll\delta\ll1$.
  We can additionally estimate
  \begin{align}
    |\Da1y^{\eps,\delta}_i-D\bar{y}(x)|&=\bigg|\,\avint_{x_i}^{x_{i+1}}\Big(
      D\bar{y}\dx-\bar{r}^\delta-D\bar{y}(x)\Big)\dx\bigg|,\notag\\
    &\leq \eps\norm{D^2\bar{y}}_\infty+\delta\norm{r}_{\ell^2(\Z)}.
    \label{eq:limsup1est2}
  \end{align}
  
  \subsubsection*{Pointwise upper bounds on $s_i^{\eps,\delta}$}
  Next, we define $s_i^{\eps,\delta}$ in a similar fashion to \eqref{eq:sidef},
  but adding and subtracting an extra $\phin(\Da2y^{\eps,\delta}_i)$ term, we
  have
  \begin{align*}
    s_i^{\eps,\delta}&= \bigg(W(\Da2y^{\eps,\delta}_i)
      -\avint_{x_i}^{x_{i+2}}W(D\bar{y})\dx\bigg)\\
    &\qquad\quad+\Big(\smfrac12\phin(\Da1y^{\eps,\delta}_i)
      +\smfrac12\phin(\Da1y^{\eps,\delta}_{i+1})
      -\phin(\Da2y_i^{\eps,\delta})\Big)\\
    &\qquad\qquad\qquad+\avint_{x_i}^{x_{i+2}}\Big(\sigma D\bar{u}-\sigma^\eps
      Du^{\eps,\delta}-\Sigma\,\big(Dy^{\eps,\delta}-D\bar{y}\big)\Big)\dx,\\
    &=:T_1+T_2+T_3.
  \end{align*}
  $T_1$ can be treated using Jensen's inequality and the convexity of $W$:
  \begin{align}
    T_1&=W(\Da2y_i^{\eps,\delta})-\avint_{x_i}^{x_{i+2}}W(D\bar{y})\dx\notag\\
      &\leq\smfrac12W(\Da1y_i^{\eps,\delta})+\smfrac12W(\Da1y^{\eps,\delta}
        _{i+1})-\avint_{x_i}^{x_{i+1}}W(D\bar{y})\dx,\notag\\
      &\leq\avint_{x_i}^{x_{i+2}}W(Dy^{\eps,\delta})-W(D\bar{y})\dx,\notag\\
    &\leq \avint_{x_i}^{x_{i+1}}W^\prime(Dy^{\eps,\delta})\big(Dy^{\eps,\delta}
      -D\bar{y}\big)\dx.\label{eq:limsup1est3}
  \end{align}
  The final integral term, $T_3$, is bounded by
  \begin{align}
    T_3&=\avint_{x_i}^{x_{i+2}}\sigma D\bar{u}-\sigma^\eps Du^{\eps,\delta}
      -\Sigma\,\big(Du^{\eps,\delta}-D\bar{u}\big)\dx\notag\\
      &=\avint_{x_i}^{x_{i+2}}(\sigma+\Sigma)\big(D\bar{u}-Du^{\eps,\delta}\big)
      +(\sigma-\sigma^\eps)Du^{\eps,\delta}\dx,\notag\\
    &\leq \avint_{x_i}^{x_{i+2}}W^\prime(D\bar{y})\big(D\bar{y}-Dy^{\eps,\delta}
      \big)\dx+C\eps^{3/2}\norm{Du^{\eps,\delta}}_2,\label{eq:limsup1est4}
  \end{align}
  using the Euler--Lagrange equations and the estimate proved in
  \eqref{eq:sigmaest3}.
  The remaining part of the expression, $T_2$, can then be estimated as follows:
  \begin{align*}
    T_2&=\smfrac12\phin(\Da1y^{\eps,\delta}_i)
      +\smfrac12\phin(\Da1y^{\eps,\delta}
      _{i+1})-\phin\big(\Da2y_i^{\eps,\delta}\big)\\
      &=\frac14\int_{\Da1y_i^{\eps,\delta}}^{\Da1y_{i+1}^{\eps,\delta}}
      \int_{\Da1y_i^{\eps,\delta}}^{\Da1y_{i+1}^{\eps,\delta}}
      \phin^{\prime\prime}\big(\smfrac{s+t}{2}\big)\ds\dt,\\
    &\leq C\,\big|\Da1y_{i+1}^{\eps,\delta}-\Da1y_i^{\eps,\delta}\big|^2,
  \end{align*}
  where we have used the fact that $\phin^{\prime\prime}$ is bounded 
  on the domain of integration for sufficiently small $\eps$ and $\delta$, and
  $\phin\in\CC^2$. If $i\in\{-K,\ldots,K-1\}$, then applying the Mean Value
  Theorem with $\xi\in(x_i,x_{i+1})$ and $\zeta\in(x_{i+1},x_{i+2})$ gives
  \begin{align}
    \smfrac12\phin(\Da1y^{\eps,\delta}_i)+\smfrac12\phin(\Da1y^{\eps,\delta}
      _{i+1})-\phin\big(\Da2y_i^{\eps,\delta}\big)&\leq C|D\bar{y}(\xi)
      -D\bar{y}(\zeta)+r_{i+1}-r_i|^2,\notag\\
    &\leq C\Big(\norm{D^2\bar{y}}_\infty\eps+|r_i|+|r_{i+1}|\Big)^2,\notag\\
    &\leq C\Big(\norm{D^2\bar{y}}^2_\infty\eps^2+|r_i|^2+|r_{i+1}|^2\Big).
      \label{eq:limsup1est5}
  \end{align}
  In the case where $i\notin\{-K,\ldots,K-1\}$, we obtain
  \begin{align*}
    \smfrac12\phin(\Da1y^{\eps,\delta}_i)
      +\smfrac12\phin(\Da1y^{\eps,\delta}_{i+1})
      -\phin\big(\Da2y_i^{\eps,\delta}\big)
      &\leq C\eps^2\norm{D^2\bar{y}}_\infty.
  \end{align*}
  Adding together \eqref{eq:limsup1est3}, \eqref{eq:limsup1est4} and
  \eqref{eq:limsup1est5}, for $i\in\{-K-1,\ldots,K-1\}$ we have
  \begin{align*}
    s_i^{\eps,\delta}&\leq\avint_{x_i}^{x_{i+2}}\Big(W^\prime(Dy^{\eps,\delta})
      -W^\prime(D\bar{y})\Big)\big(Dy^{\eps,\delta}-D\bar{y}\big)\dx+C\eps^{3/2}
      \norm{Du^{\eps,\delta}}_2\\
    &\qquad\qquad+C\eps^2+C\big(|r_i|^2+|r_{i+1}|^2\big),\\
    &\leq C\avint_{x_i}^{x_{i+2}}|Dy^{\eps,\delta}-D\bar{y}|^2\dx
      +C\Big(\eps^{3/2}+\eps^2+|r_i|^2+|r_{i+1}|^2\Big),\\
    &\leq C\,\Big(\big(\eps+\delta\big)^2+\eps^{3/2}
      +\eps^2+|r_i|^2+|r_{i+1}|^2\Big),\\
    &\leq C\,\Big(\delta^2+\eps^{3/2}+\eps^2+|r_i|^2+|r_{i+1}|^2\Big),
  \end{align*}
  using $W^\prime\in\CC^1$, estimate \eqref{eq:limsup1est2} and Jensen's
  inequality. For the other indices, we obtain
  \begin{align*}
    s_i^{\eps,\delta}&\leq\avint_{x_i}^{x_{i+2}}\Big(W^\prime(Dy^{\eps,\delta})
      -W^\prime(D\bar{y})\Big)\big(Dy^{\eps,\delta}-D\bar{y}\big)\dx+C\eps^{3/2}
      \norm{Du^{\eps,\delta}}_2,\\
    &\leq C\,\big(\eps^2+\eps^{3/2}\big).
  \end{align*}
  
  \subsubsection*{Conclusion of the argument}
  The two pointwise estimates just obtained imply
  \begin{align*}
    \sum_{i=-\infty}^{\infty}s_i^{\eps,\delta}&\leq
      C\sum_{i=-N}^{N-1}\eps^{3/2}+C\sum_{i=-K}^{K}\big(|r_i|^2+\delta^2\big),\\
    &\leq C\Big(\eps^{1/2}+\norm{r}_{\ell^2(\Z)}^2+\delta\Big).
  \end{align*}
  By choosing $K=\lfloor \sqrt{N}\rfloor$, we have that
  $\delta\leq C\eps^{1/2}$, and so an application of Fatou's Lemma with the
  pointwise upper bound we have just proven implies that
  \begin{align*}
    \limsup_{\eps\conv0}\sum_{i=-\infty}^{\infty}s_i^{\eps,\delta}+\Edefone(Dy^\eps)&\leq
      \sum_{i=-\infty}^{\infty}\limsup_{\eps\conv0}s_i^{\eps,\delta}+\lim_{\eps\conv0}\Edefone(Dy^\eps),\\
      &=\Einf(r).
  \end{align*}
  This now proves Proposition \ref{prop:1limsup}, and concludes the proof of Theorem
  \ref{thm:1Gconv}.
  
  \begin{remark}
  This result shows that the perturbation to the minimiser from the continuum
  model is confined to an exponentially thin boundary layer. Note that the
  linearisation of the functional $\Einf$ in Section \ref{subsec:formal} yielded
  a similar solution structure; this exponential decay suggests that any
  interaction between defects of the type described here is likely to `decouple'
  if one were to study a situation in which there were multiple
  defects of a fixed and finite number which are well-separated in the limit
  $\eps\conv0$.
  \end{remark}
  
  \section*{Conclusion}
  We have presented an analysis of a model for a point defect in a 1D chain of
  atoms interacting under assumptions which attempt to replicate a
  Lennard-Jones type interactions in an elastic regime. We have derived the
  $0^\mathrm{th}$-order $\G$-limit, which is identical to the limit when there
  is no defect.
  
  We then proved that the $1^\mathrm{st}$-order $\G$-limit exists and have
  given an explicit characterisation of this limit in terms of an infinite cell
  problem, and shown that the perturbation introduced by the defect is
  confined to an exponentially thin boundary layer.
  
  \section*{Acknowledgements}
  This work was supported by a studentship which was granted as part of the
  EPSRC Science and Innovation award to the Oxford Centre for Nonlinear PDE
  (EP/E035027/1).
  
  The author would like to thank Christoph Ortner for proposing the Confined
  Lennard-Jones model described above and for a great many discussions 
  throughout this project.
  
%  This result will allow us to show the order of the next term in the $\G$-
%  expansion.
%  
%  \section{Completeness of the $\G$-expansion}
%  In this section, we use the results of Section \ref{sec:props} to obtain a
%  completeness result for the $\G$-expansion; we show in particular that the
%  expansion is complete up to $O(\eps^2)$. However, in the particular case
%  where $f\equiv0$, we show that the solution is complete up to an
%  asymptotically small error.
%  
%  \begin{theorem}
%  Defining $\Ee_{1+\alpha}:\AN\conv\R\cup\{+\infty\}$ to be
%  \begin{align*}
%    \Ee_{1+\alpha}(y^\eps):=\frac{\Ee(y^\eps)-\Ec(\bar{y}_F)
%    -\eps E_1(\bar{y}_F)}{\eps^{1+\alpha}},
%  \end{align*}
%  we have that for $\alpha\in(0,1)$,
%  \begin{align*}
%    \Glim_{\eps\conv0}\Ee_{1+\alpha}(y) = \begin{cases}
%      0       &y=\bar{y}_F,\\
%      +\infty &\text{otherwise}.
%    \end{cases}
%  \end{align*}
%  In other words, the $\G$-expansion is complete up to order $O(\eps^2)$.
%  Further, if $f\equiv0$, the expansion is complete up to any algebraic order;
%  that is, order $o(\eps^\alpha)$ for any $\alpha\in\R^+$.
%  \end{theorem}
%  
%  In order to prove this theorem, we will use the following coercivity result:
%  
%  \begin{lemma}
%  If $\Ee_{1+\alpha}(y^\eps)$ is uniformly bounded, then $\mathcal{P}_\eps y^\eps$
%  converges strongly to $r$ in $\ell^2(\Z)$.
%  \end{lemma}
%  
%  \begin{proof}
%  The proof of this lemma arises from estimate
%  \begin{align*}
%    s_i^\eps \geq \avint_{x_i}^{x_{i+2}}(\sigma-\sigma^\eps)\cdot Dy^\eps +c|
%  \end{align*}
%  \end{proof}

  \bibliography{qc}
  \bibliographystyle{alpha}
  
\end{document}